\documentclass[12pt]{amsart}
\usepackage[latin1]{inputenc}
\usepackage{amsbsy}
\usepackage{amscd} 
\usepackage{amsfonts}
\usepackage{amsgen} 
\usepackage{amsmath}
\usepackage{amsopn} 
\usepackage{amssymb}
\usepackage{amstext}
\usepackage{amsthm} 
\usepackage{amsxtra}
\usepackage[all]{xy}
\usepackage{epsfig}
\usepackage{mathtools}
\usepackage{color}
 \definecolor{hellgrau}{gray}{0.9}
 \definecolor{hgrau}{gray}{0.5} 
\usepackage{rotating}

\usepackage{calc}
\newcounter{PartitionDepth}
\newcounter{PartitionLength}
\newcommand{\parti}[2]{
 \begin{picture}(#2,#1)
 \setcounter{PartitionDepth}{-1-#1}
 \put(#2,\thePartitionDepth){\line(0,1){#1}}
 \end{picture}}
\newcommand{\partii}[3]{
 \begin{picture}(#3,#1)
 \setcounter{PartitionLength}{#3-#2}
 \setcounter{PartitionDepth}{-1-#1}
 \put(#2,\thePartitionDepth){\line(0,1){#1}}     
 \put(#3,\thePartitionDepth){\line(0,1){#1}}
 \put(#2,\thePartitionDepth){\line(1,0){\thePartitionLength}}
 \end{picture}}

\newcommand{\partiv}[5]{
 \begin{picture}(#5,#1)
 \setcounter{PartitionLength}{#5-#2}
 \setcounter{PartitionDepth}{-1-#1}
 \put(#2,\thePartitionDepth){\line(0,1){#1}}
 \put(#3,\thePartitionDepth){\line(0,1){#1}}
 \put(#4,\thePartitionDepth){\line(0,1){#1}}
 \put(#5,\thePartitionDepth){\line(0,1){#1}}
 \put(#2,\thePartitionDepth){\line(1,0){\thePartitionLength}} 
 \end{picture}}

\newcommand{\upparti}[2]{
 \begin{picture}(#2,#1)
 \setcounter{PartitionDepth}{#1}
 \put(#2,0){\line(0,1){#1}}
 \end{picture}}
\newcommand{\uppartii}[3]{
 \begin{picture}(#3,#1)
 \setcounter{PartitionLength}{#3-#2}
 \setcounter{PartitionDepth}{#1}
 \put(#2,0){\line(0,1){#1}}     
 \put(#3,0){\line(0,1){#1}}
 \put(#2,\thePartitionDepth){\line(1,0){\thePartitionLength}}
 \end{picture}}
\newcommand{\uppartiii}[4]{
 \begin{picture}(#4,#1)
 \setcounter{PartitionLength}{#4-#2}
 \setcounter{PartitionDepth}{#1}
 \put(#2,0){\line(0,1){#1}}
 \put(#3,0){\line(0,1){#1}}
 \put(#4,0){\line(0,1){#1}}
 \put(#2,\thePartitionDepth){\line(1,0){\thePartitionLength}} 
 \end{picture}}
\newcommand{\uppartiv}[5]{
 \begin{picture}(#5,#1)
 \setcounter{PartitionLength}{#5-#2}
 \setcounter{PartitionDepth}{#1}
 \put(#2,0){\line(0,1){#1}}
 \put(#3,0){\line(0,1){#1}}
 \put(#4,0){\line(0,1){#1}}
 \put(#5,0){\line(0,1){#1}}
 \put(#2,\thePartitionDepth){\line(1,0){\thePartitionLength}} 
 \end{picture}}

\newsavebox{\boxpaarpart}
   \savebox{\boxpaarpart}
   { \begin{picture}(0.7,0.5)
     \put(0,0){\line(0,1){0.4}}
     \put(0.5,0){\line(0,1){0.4}}
     \put(0,0.4){\line(1,0){0.5}}
     \end{picture}}
\newsavebox{\boxbaarpart}
   \savebox{\boxbaarpart}
   { \begin{picture}(0.7,0.5)
     \put(0,0){\line(0,1){0.4}}
     \put(0.5,0){\line(0,1){0.4}}
     \put(0,0){\line(1,0){0.5}}
     \end{picture}}
\newsavebox{\boxdreipart}
   \savebox{\boxdreipart}
   { \begin{picture}(1,0.5)
     \put(0,0){\line(0,1){0.4}}
     \put(0.4,0){\line(0,1){0.4}}
     \put(0.8,0){\line(0,1){0.4}}
     \put(0,0.4){\line(1,0){0.8}}
     \end{picture}}
\newsavebox{\boxvierpart}
   \savebox{\boxvierpart}
   { \begin{picture}(1.4,0.5)
     \put(0,0){\line(0,1){0.4}}
     \put(0.4,0){\line(0,1){0.4}}
     \put(0.8,0){\line(0,1){0.4}}
     \put(1.2,0){\line(0,1){0.4}}
     \put(0,0.4){\line(1,0){1.2}}
     \end{picture}}
\newsavebox{\boxvierpartrot}
   \savebox{\boxvierpartrot}
   { \begin{picture}(0.5,0.8)
     \put(0,-0.2){\line(0,1){0.3}}
     \put(0.4,-0.2){\line(0,1){0.3}}
     \put(0,0.1){\line(1,0){0.4}}
     \put(0,0.4){\line(0,1){0.3}}
     \put(0.4,0.4){\line(0,1){0.3}}
     \put(0,0.4){\line(1,0){0.4}}
     \put(0.2,0.1){\line(0,1){0.3}}
     \end{picture}}
\newsavebox{\boxvierpartrotdrei}
   \savebox{\boxvierpartrotdrei}
   { \begin{picture}(1,0.8)
     \put(0,-0.2){\line(0,1){0.3}}
     \put(0.4,-0.2){\line(0,1){0.8}}
     \put(0.8,-0.2){\line(0,1){0.3}}
     \put(0,0.1){\line(1,0){0.8}}
     \end{picture}}
\newsavebox{\boxcrosspart}
   \savebox{\boxcrosspart}
   { \begin{picture}(0.5,0.8)
     \put(0,-0.2){\line(1,2){0.4}}
     \put(0.4,-0.2){\line(-1,2){0.4}}
     \end{picture}}
\newsavebox{\boxhalflibpart}
   \savebox{\boxhalflibpart}
   { \begin{picture}(1,0.8)
     \put(0,-0.2){\line(1,1){0.8}}
     \put(0.8,-0.2){\line(-1,1){0.8}}
     \put(0.4,-0.2){\line(0,1){0.8}}
     \end{picture}}
\newsavebox{\boxpositioner} 
   \savebox{\boxpositioner}
   { \begin{picture}(1.4,0.5)
     \put(0,0){\line(0,1){0.3}}
     \put(0.4,0){\line(0,1){0.6}}
     \put(0.8,0){\line(0,1){0.3}}
     \put(1.2,0){\line(0,1){0.6}}
     \put(0.4,0.6){\line(1,0){0.8}}
     \end{picture}}
\newsavebox{\boxfatcross} 
   \savebox{\boxfatcross}
   { \begin{picture}(1.4,1.3)
     \put(0,-0.2){\line(0,1){0.3}}
     \put(0.4,-0.2){\line(0,1){0.3}}
     \put(0.8,-0.2){\line(0,1){0.3}}
     \put(1.2,-0.2){\line(0,1){0.3}}
     \put(0,0.1){\line(1,0){0.4}}
     \put(0.8,0.1){\line(1,0){0.4}}
     \put(0,0.9){\line(0,1){0.3}}
     \put(0.4,0.9){\line(0,1){0.3}}
     \put(0.8,0.9){\line(0,1){0.3}}
     \put(1.2,0.9){\line(0,1){0.3}}
     \put(0,0.9){\line(1,0){0.4}}
     \put(0.8,0.9){\line(1,0){0.4}}
     \put(0.2,0.1){\line(1,1){0.8}}
     \put(0.2,0.9){\line(1,-1){0.8}}
     \end{picture}}
\newsavebox{\boxprimarypart} 
   \savebox{\boxprimarypart}
   { \begin{picture}(1,1.3)
     \put(0,-0.2){\line(0,1){0.3}}
     \put(0.4,-0.2){\line(0,1){0.3}}
     \put(0.8,-0.2){\line(0,1){0.3}}
     \put(0.4,0.1){\line(1,0){0.4}}
     \put(0,0.9){\line(0,1){0.3}}
     \put(0.4,0.9){\line(0,1){0.3}}
     \put(0.8,0.9){\line(0,1){0.3}}
     \put(0,0.9){\line(1,0){0.4}}
     \put(0,0.1){\line(1,1){0.8}}
     \put(0.2,0.9){\line(1,-2){0.4}}
     \end{picture}}

\newsavebox{\boxidpartww}
   \savebox{\boxidpartww}
   { \begin{picture}(0.5,1.2)
     \put(0.2,0.15){\line(0,1){0.7}}
     \put(0,-0.2){$\circ$}
     \put(0,0.8){$\circ$}
     \end{picture}}
\newsavebox{\boxidpartbw}
   \savebox{\boxidpartbw}
   { \begin{picture}(0.5,1.2)
     \put(0.2,0.15){\line(0,1){0.7}}
     \put(0,-0.2){$\circ$}
     \put(0,0.8){$\bullet$}
     \end{picture}}
\newsavebox{\boxidpartwb}
   \savebox{\boxidpartwb}
   { \begin{picture}(0.5,1.2)
     \put(0.2,0.15){\line(0,1){0.7}}
     \put(0,-0.2){$\bullet$}
     \put(0,0.8){$\circ$}
     \end{picture}}
\newsavebox{\boxidpartbb}
   \savebox{\boxidpartbb}
   { \begin{picture}(0.5,1.2)
     \put(0.2,0.15){\line(0,1){0.7}}
     \put(0,-0.2){$\bullet$}
     \put(0,0.8){$\bullet$}
     \end{picture}}

\newsavebox{\boxidpartsingletonbb}
   \savebox{\boxidpartsingletonbb}
   { \begin{picture}(0.5,1.2)
     \put(0.2,0.15){\line(0,1){0.2}}
     \put(0.2,0.65){\line(0,1){0.2}}
     \put(0,-0.2){$\bullet$}
     \put(0,0.8){$\bullet$}
     \end{picture}}
\newsavebox{\boxidpartsingletonww}
   \savebox{\boxidpartsingletonww}
   { \begin{picture}(0.5,1.2)
     \put(0.2,0.15){\line(0,1){0.2}}
     \put(0.2,0.65){\line(0,1){0.2}}
     \put(0,-0.2){$\circ$}
     \put(0,0.8){$\circ$}
     \end{picture}}
     
\newsavebox{\boxpaarpartbb}
   \savebox{\boxpaarpartbb}
   { \begin{picture}(1,1)
     \put(0.2,0.2){\line(0,1){0.4}}
     \put(0.7,0.2){\line(0,1){0.4}}
     \put(0.2,0.6){\line(1,0){0.5}}
     \put(0.5,-0.2){$\bullet$}
     \put(0,-0.2){$\bullet$}
     \end{picture}}
\newsavebox{\boxpaarpartww}
   \savebox{\boxpaarpartww}
   { \begin{picture}(1,1)
     \put(0.2,0.2){\line(0,1){0.4}}
     \put(0.7,0.2){\line(0,1){0.4}}
     \put(0.2,0.6){\line(1,0){0.5}}
     \put(0.5,-0.2){$\circ$}
     \put(0,-0.2){$\circ$}
     \end{picture}}
\newsavebox{\boxpaarpartbw}
   \savebox{\boxpaarpartbw}
   { \begin{picture}(1,1)
     \put(0.2,0.2){\line(0,1){0.4}}
     \put(0.7,0.2){\line(0,1){0.4}}
     \put(0.2,0.6){\line(1,0){0.5}}
     \put(0.5,-0.2){$\circ$}
     \put(0,-0.2){$\bullet$}
     \end{picture}}
\newsavebox{\boxpaarpartwb}
   \savebox{\boxpaarpartwb}
   { \begin{picture}(1,1)
     \put(0.2,0.2){\line(0,1){0.4}}
     \put(0.7,0.2){\line(0,1){0.4}}
     \put(0.2,0.6){\line(1,0){0.5}}
     \put(0.5,-0.2){$\bullet$}
     \put(0,-0.2){$\circ$}
     \end{picture}}
\newsavebox{\boxbaarpartbb}
   \savebox{\boxbaarpartbb}
   { \begin{picture}(1,1)
     \put(0.2,-0.1){\line(0,1){0.4}}
     \put(0.7,-0.1){\line(0,1){0.4}}
     \put(0.2,-0.1){\line(1,0){0.5}}
     \put(0.5,0.3){$\bullet$}
     \put(0,0.3){$\bullet$}
     \end{picture}}
\newsavebox{\boxbaarpartww}
   \savebox{\boxbaarpartww}
   { \begin{picture}(1,1)
     \put(0.2,-0.1){\line(0,1){0.4}}
     \put(0.7,-0.1){\line(0,1){0.4}}
     \put(0.2,-0.1){\line(1,0){0.5}}
     \put(0.5,0.3){$\circ$}
     \put(0,0.3){$\circ$}
     \end{picture}}
\newsavebox{\boxbaarpartbw}
   \savebox{\boxbaarpartbw}
   { \begin{picture}(1,1)
     \put(0.2,-0.1){\line(0,1){0.4}}
     \put(0.7,-0.1){\line(0,1){0.4}}
     \put(0.2,-0.1){\line(1,0){0.5}}
     \put(0.5,0.3){$\circ$}
     \put(0,0.3){$\bullet$}
     \end{picture}}
\newsavebox{\boxbaarpartwb}
   \savebox{\boxbaarpartwb}
   { \begin{picture}(1,1)
     \put(0.2,-0.1){\line(0,1){0.4}}
     \put(0.7,-0.1){\line(0,1){0.4}}
     \put(0.2,-0.1){\line(1,0){0.5}}
     \put(0.5,0.3){$\bullet$}
     \put(0,0.3){$\circ$}
     \end{picture}}
\newsavebox{\boxpaarbaarpartwwww}
   \savebox{\boxpaarbaarpartwwww}
   { \begin{picture}(1,1.2)
     \put(0.2,0.15){\line(0,1){0.2}}
     \put(0.2,0.65){\line(0,1){0.2}}
     \put(0.2,0.65){\line(1,0){0.5}}
     \put(0.2,0.35){\line(1,0){0.5}}
     \put(0.7,0.15){\line(0,1){0.2}}
     \put(0.7,0.65){\line(0,1){0.2}}
     \put(0,0.8){$\circ$}
     \put(0.5,0.8){$\circ$}
     \put(0,-0.2){$\circ$}
     \put(0.5,-0.2){$\circ$}
     \end{picture}}     

\newsavebox{\boxsingletonw}
   \savebox{\boxsingletonw}
   { \begin{picture}(0.5, 1)
     \put(0,0.3){$\uparrow$}
     \put(0,-0.2){$\circ$}
     \end{picture}}
\newsavebox{\boxsingletonb}
   \savebox{\boxsingletonb}
   { \begin{picture}(0.5, 1)
     \put(0,0.3){$\uparrow$}
     \put(0,-0.2){$\bullet$}
     \end{picture}}
\newsavebox{\boxdownsingletonw}
   \savebox{\boxdownsingletonw}
   { \begin{picture}(0.5, 1)
     \put(0,0){$\downarrow$}
     \put(0,0.6){$\circ$}
     \end{picture}}
\newsavebox{\boxdownsingletonb}
   \savebox{\boxdownsingletonb}
   { \begin{picture}(0.5, 1)
     \put(0,0){$\downarrow$}
     \put(0,0.6){$\bullet$}
     \end{picture}}

\newsavebox{\boxvierpartwwww}
   \savebox{\boxvierpartwwww}
   { \begin{picture}(1.7,1)
     \put(0.2,0.2){\line(0,1){0.4}}
     \put(0.6,0.2){\line(0,1){0.4}}
     \put(1,0.2){\line(0,1){0.4}}
     \put(1.4,0.2){\line(0,1){0.4}}
     \put(0.2,0.6){\line(1,0){1.2}}
     \put(0,-0.2){$\circ$}
     \put(0.4,-0.2){$\circ$}
     \put(0.8,-0.2){$\circ$}
     \put(1.2,-0.2){$\circ$}
     \end{picture}}
\newsavebox{\boxdownvierpartwwww}
   \savebox{\boxdownvierpartwwww}
   { \begin{picture}(1.7,1)
     \put(0.2,0.2){\line(0,1){0.4}}
     \put(0.6,0.2){\line(0,1){0.4}}
     \put(1,0.2){\line(0,1){0.4}}
     \put(1.4,0.2){\line(0,1){0.4}}
     \put(0.2,0.2){\line(1,0){1.2}}
     \put(0,0.5){$\circ$}
     \put(0.4,0.5){$\circ$}
     \put(0.8,0.5){$\circ$}
     \put(1.2,0.5){$\circ$}
     \end{picture}}     
\newsavebox{\boxvierpartwbwb}
   \savebox{\boxvierpartwbwb}
   { \begin{picture}(1.7,1)
     \put(0.2,0.2){\line(0,1){0.4}}
     \put(0.6,0.2){\line(0,1){0.4}}
     \put(1,0.2){\line(0,1){0.4}}
     \put(1.4,0.2){\line(0,1){0.4}}
     \put(0.2,0.6){\line(1,0){1.2}}
     \put(0,-0.2){$\circ$}
     \put(0.4,-0.2){$\bullet$}
     \put(0.8,-0.2){$\circ$}
     \put(1.2,-0.2){$\bullet$}
     \end{picture}}
\newsavebox{\boxvierpartbwbw}
   \savebox{\boxvierpartbwbw}
   { \begin{picture}(1.7,1)
     \put(0.2,0.2){\line(0,1){0.4}}
     \put(0.6,0.2){\line(0,1){0.4}}
     \put(1,0.2){\line(0,1){0.4}}
     \put(1.4,0.2){\line(0,1){0.4}}
     \put(0.2,0.6){\line(1,0){1.2}}
     \put(0,-0.2){$\bullet$}
     \put(0.4,-0.2){$\circ$}
     \put(0.8,-0.2){$\bullet$}
     \put(1.2,-0.2){$\circ$}
     \end{picture}}
\newsavebox{\boxvierpartwwbb}
   \savebox{\boxvierpartwwbb}
   { \begin{picture}(1.7,1)
     \put(0.2,0.2){\line(0,1){0.4}}
     \put(0.6,0.2){\line(0,1){0.4}}
     \put(1,0.2){\line(0,1){0.4}}
     \put(1.4,0.2){\line(0,1){0.4}}
     \put(0.2,0.6){\line(1,0){1.2}}
     \put(0,-0.2){$\circ$}
     \put(0.4,-0.2){$\circ$}
     \put(0.8,-0.2){$\bullet$}
     \put(1.2,-0.2){$\bullet$}
     \end{picture}}
\newsavebox{\boxvierpartwbbw}
   \savebox{\boxvierpartwbbw}
   { \begin{picture}(1.7,1)
     \put(0.2,0.2){\line(0,1){0.4}}
     \put(0.6,0.2){\line(0,1){0.4}}
     \put(1,0.2){\line(0,1){0.4}}
     \put(1.4,0.2){\line(0,1){0.4}}
     \put(0.2,0.6){\line(1,0){1.2}}
     \put(0,-0.2){$\circ$}
     \put(0.4,-0.2){$\bullet$}
     \put(0.8,-0.2){$\bullet$}
     \put(1.2,-0.2){$\circ$}
     \end{picture}}
\newsavebox{\boxvierpartbwwb}
   \savebox{\boxvierpartbwwb}
   { \begin{picture}(1.7,1)
     \put(0.2,0.2){\line(0,1){0.4}}
     \put(0.6,0.2){\line(0,1){0.4}}
     \put(1,0.2){\line(0,1){0.4}}
     \put(1.4,0.2){\line(0,1){0.4}}
     \put(0.2,0.6){\line(1,0){1.2}}
     \put(0,-0.2){$\bullet$}
     \put(0.4,-0.2){$\circ$}
     \put(0.8,-0.2){$\circ$}
     \put(1.2,-0.2){$\bullet$}
     \end{picture}}
\newsavebox{\boxvierpartrotwbwb}
   \savebox{\boxvierpartrotwbwb}
   { \begin{picture}(1,1.2)
     \put(0.2,0.15){\line(0,1){0.2}}
     \put(0.2,0.65){\line(0,1){0.2}}
     \put(0.2,0.65){\line(1,0){0.5}}
     \put(0.2,0.35){\line(1,0){0.5}}
     \put(0.45,0.35){\line(0,1){0.3}}
     \put(0.7,0.15){\line(0,1){0.2}}
     \put(0.7,0.65){\line(0,1){0.2}}
     \put(0,0.8){$\circ$}
     \put(0.5,0.8){$\bullet$}
     \put(0,-0.2){$\circ$}
     \put(0.5,-0.2){$\bullet$}
     \end{picture}}
\newsavebox{\boxvierpartrotbwwb}
   \savebox{\boxvierpartrotbwwb}
   { \begin{picture}(1,1.2)
     \put(0.2,0.15){\line(0,1){0.2}}
     \put(0.2,0.65){\line(0,1){0.2}}
     \put(0.2,0.65){\line(1,0){0.5}}
     \put(0.2,0.35){\line(1,0){0.5}}
     \put(0.45,0.35){\line(0,1){0.3}}
     \put(0.7,0.15){\line(0,1){0.2}}
     \put(0.7,0.65){\line(0,1){0.2}}
     \put(0,0.8){$\bullet$}
     \put(0.5,0.8){$\circ$}
     \put(0,-0.2){$\circ$}
     \put(0.5,-0.2){$\bullet$}
     \end{picture}}
\newsavebox{\boxvierpartrotbwbw}
   \savebox{\boxvierpartrotbwbw}
   { \begin{picture}(1,1.2)
     \put(0.2,0.15){\line(0,1){0.2}}
     \put(0.2,0.65){\line(0,1){0.2}}
     \put(0.2,0.65){\line(1,0){0.5}}
     \put(0.2,0.35){\line(1,0){0.5}}
     \put(0.45,0.35){\line(0,1){0.3}}
     \put(0.7,0.15){\line(0,1){0.2}}
     \put(0.7,0.65){\line(0,1){0.2}}
     \put(0,0.8){$\bullet$}
     \put(0.5,0.8){$\circ$}
     \put(0,-0.2){$\bullet$}
     \put(0.5,-0.2){$\circ$}
     \end{picture}}
\newsavebox{\boxvierpartrotwwww}
   \savebox{\boxvierpartrotwwww}
   { \begin{picture}(1,1.2)
     \put(0.2,0.15){\line(0,1){0.2}}
     \put(0.2,0.65){\line(0,1){0.2}}
     \put(0.2,0.65){\line(1,0){0.5}}
     \put(0.2,0.35){\line(1,0){0.5}}
     \put(0.45,0.35){\line(0,1){0.3}}
     \put(0.7,0.15){\line(0,1){0.2}}
     \put(0.7,0.65){\line(0,1){0.2}}
     \put(0,0.8){$\circ$}
     \put(0.5,0.8){$\circ$}
     \put(0,-0.2){$\circ$}
     \put(0.5,-0.2){$\circ$}
     \end{picture}}
\newsavebox{\boxvierpartrotbbbb}
   \savebox{\boxvierpartrotbbbb}
   { \begin{picture}(1,1.2)
     \put(0.2,0.15){\line(0,1){0.2}}
     \put(0.2,0.65){\line(0,1){0.2}}
     \put(0.2,0.65){\line(1,0){0.5}}
     \put(0.2,0.35){\line(1,0){0.5}}
     \put(0.45,0.35){\line(0,1){0.3}}
     \put(0.7,0.15){\line(0,1){0.2}}
     \put(0.7,0.65){\line(0,1){0.2}}
     \put(0,0.8){$\bullet$}
     \put(0.5,0.8){$\bullet$}
     \put(0,-0.2){$\bullet$}
     \put(0.5,-0.2){$\bullet$}
     \end{picture}}

\newsavebox{\boxdreipartwww}
   \savebox{\boxdreipartwww}
   { \begin{picture}(1.2,1)
     \put(0.2,0.2){\line(0,1){0.4}}
     \put(0.6,0.2){\line(0,1){0.4}}
     \put(1,0.2){\line(0,1){0.4}}
     \put(0.2,0.6){\line(1,0){0.8}}
     \put(0,-0.2){$\circ$}
     \put(0.4,-0.2){$\circ$}
     \put(0.8,-0.2){$\circ$}
     \end{picture}}
\newsavebox{\boxdowndreipartwww}
   \savebox{\boxdowndreipartwww}
   { \begin{picture}(1.2,1)
     \put(0.2,-0.1){\line(0,1){0.4}}
     \put(0.6,-0.1){\line(0,1){0.4}}
     \put(1,-0.1){\line(0,1){0.4}}
     \put(0.2,-0.1){\line(1,0){0.8}}
     \put(0,0.3){$\circ$}
     \put(0.4,0.3){$\circ$}
     \put(0.8,0.3){$\circ$}
     \end{picture}}     
\newsavebox{\boxdreipartrotwww}
   \savebox{\boxdreipartrotwww}
   { \begin{picture}(1,1.2)
     \put(0.2,0.65){\line(0,1){0.2}}
     \put(0.2,0.65){\line(1,0){0.5}}
     \put(0.45,0.15){\line(0,1){0.5}}
     \put(0.7,0.65){\line(0,1){0.2}}
     \put(0,0.8){$\circ$}
     \put(0.5,0.8){$\circ$}
     \put(0.25,-0.2){$\circ$}
     \end{picture}}     
\newsavebox{\boxdowndreipartrotwww}
   \savebox{\boxdowndreipartrotwww}
   { \begin{picture}(1,1.2)
     \put(0.2,0.15){\line(0,1){0.2}}
     \put(0.2,0.35){\line(1,0){0.5}}
     \put(0.45,0.35){\line(0,1){0.5}}
     \put(0.7,0.15){\line(0,1){0.2}}
     \put(0,-0.2){$\circ$}
     \put(0.5,-0.2){$\circ$}
     \put(0.25,0.8){$\circ$}
     \end{picture}}          

\newsavebox{\boxsechspartwbwbwb}
   \savebox{\boxsechspartwbwbwb}
   { \begin{picture}(2.5,1)
     \put(0.2,0.2){\line(0,1){0.4}}
     \put(0.6,0.2){\line(0,1){0.4}}
     \put(1,0.2){\line(0,1){0.4}}
     \put(1.4,0.2){\line(0,1){0.4}}
     \put(1.8,0.2){\line(0,1){0.4}}
     \put(2.2,0.2){\line(0,1){0.4}}
     \put(0.2,0.6){\line(1,0){2}}
     \put(0,-0.2){$\circ$}
     \put(0.4,-0.2){$\bullet$}
     \put(0.8,-0.2){$\circ$}
     \put(1.2,-0.2){$\bullet$}
     \put(1.6,-0.2){$\circ$}
     \put(2.0,-0.2){$\bullet$}
     \end{picture}}
 
\newsavebox{\boxcrosspartwbbw}
   \savebox{\boxcrosspartwbbw}
   { \begin{picture}(1,1.4)
     \put(0.25,0.15){\line(1,2){0.4}}
     \put(0.65,0.15){\line(-1,2){0.4}}
     \put(0,0.9){$\circ$}
     \put(0.5,0.9){$\bullet$}
     \put(0,-0.2){$\bullet$}
     \put(0.5,-0.2){$\circ$}
     \end{picture}}
\newsavebox{\boxcrosspartbwwb}
   \savebox{\boxcrosspartbwwb}
   { \begin{picture}(1,1.4)
     \put(0.25,0.15){\line(1,2){0.4}}
     \put(0.65,0.15){\line(-1,2){0.4}}
     \put(0,0.9){$\bullet$}
     \put(0.5,0.9){$\circ$}
     \put(0,-0.2){$\circ$}
     \put(0.5,-0.2){$\bullet$}
     \end{picture}}
\newsavebox{\boxcrosspartwwww}
   \savebox{\boxcrosspartwwww}
   { \begin{picture}(1,1.4)
     \put(0.25,0.15){\line(1,2){0.4}}
     \put(0.65,0.15){\line(-1,2){0.4}}
     \put(0,0.9){$\circ$}
     \put(0.5,0.9){$\circ$}
     \put(0,-0.2){$\circ$}
     \put(0.5,-0.2){$\circ$}
     \end{picture}}
\newsavebox{\boxcrosspartbbbb}
   \savebox{\boxcrosspartbbbb}
   { \begin{picture}(1,1.4)
     \put(0.25,0.15){\line(1,2){0.4}}
     \put(0.65,0.15){\line(-1,2){0.4}}
     \put(0,0.9){$\bullet$}
     \put(0.5,0.9){$\bullet$}
     \put(0,-0.2){$\bullet$}
     \put(0.5,-0.2){$\bullet$}
     \end{picture}}
\newsavebox{\boxcrosspartonelinewwww}
   \savebox{\boxcrosspartonelinewwww}
   { \begin{picture}(1.7,1)
     \put(0.2,0.2){\line(0,1){0.4}}
     \put(0.6,0.2){\line(0,1){0.6}}
     \put(1,0.2){\line(0,1){0.4}}
     \put(1.4,0.2){\line(0,1){0.6}}
     \put(0.2,0.6){\line(1,0){0.8}}
     \put(0.6,0.8){\line(1,0){0.8}}     
     \put(0,-0.2){$\circ$}
     \put(0.4,-0.2){$\circ$}
     \put(0.8,-0.2){$\circ$}
     \put(1.2,-0.2){$\circ$}
     \end{picture}}     

\newsavebox{\boxhalflibpartwwwwww}
   \savebox{\boxhalflibpartwwwwww}
   { \begin{picture}(1.5,1.4)
     \put(0.3,0.15){\line(1,1){0.8}}
     \put(1.1,0.15){\line(-1,1){0.8}}
     \put(0.7,0.15){\line(0,1){0.8}}     
     \put(0,0.9){$\circ$}
     \put(0.5,0.9){$\circ$}
     \put(1,0.9){$\circ$}
     \put(0,-0.2){$\circ$}
     \put(0.5,-0.2){$\circ$}
     \put(1,-0.2){$\circ$}     
     \end{picture}}

\newsavebox{\boxpositionerd} 
   \savebox{\boxpositionerd}
   { \begin{picture}(2.6,1.5)
     \put(0.9,0.2){\line(0,1){0.9}}
     \put(2.3,0.2){\line(0,1){0.9}}
     \put(0.9,1.1){\line(1,0){1.4}}
     \put(0.05,-0.05){$^\uparrow$}
     \put(0.2,0.4){\tiny$^{\otimes d}$}
     \put(1.35,-0.05){$^\uparrow$}
     \put(1.5,0.4){\tiny$^{\otimes d}$}
     \put(0,-0.2){$\circ$}
     \put(0.7,-0.2){$\circ$}
     \put(1.3,-0.2){$\bullet$}
     \put(2.1,-0.2){$\bullet$}
     \end{picture}}
\newsavebox{\boxpositionerrpluseins} 
   \savebox{\boxpositionerrpluseins}
   { \begin{picture}(3.9,1.5)
     \put(1.6,0.2){\line(0,1){0.9}}
     \put(3.6,0.2){\line(0,1){0.9}}
     \put(1.6,1.1){\line(1,0){2}}
     \put(0.05,-0.05){$^\uparrow$}
     \put(0.2,0.4){\tiny$^{\otimes r+1}$}
     \put(2.05,-0.05){$^\uparrow$}
     \put(2.2,0.4){\tiny$^{\otimes r-1}$}
     \put(0,-0.2){$\circ$}
     \put(1.4,-0.2){$\bullet$}
     \put(2,-0.2){$\bullet$}
     \put(3.4,-0.2){$\bullet$}
     \end{picture}}
\newsavebox{\boxpositionerdt} 
   \savebox{\boxpositionerdt}
   { \begin{picture}(2.9,1.5)
     \put(1.2,0.2){\line(0,1){0.9}}
     \put(2.9,0.2){\line(0,1){0.9}}
     \put(1.2,1.1){\line(1,0){1.7}}
     \put(0.05,-0.05){$^\uparrow$}
     \put(0.2,0.4){\tiny$^{\otimes dt}$}
     \put(1.65,-0.05){$^\uparrow$}
     \put(1.8,0.4){\tiny$^{\otimes dt}$}
     \put(0,-0.2){$\circ$}
     \put(1,-0.2){$\circ$}
     \put(1.6,-0.2){$\bullet$}
     \put(2.7,-0.2){$\bullet$}
     \end{picture}}
\newsavebox{\boxpositioners} 
   \savebox{\boxpositioners}
   { \begin{picture}(2.6,1.5)
     \put(0.9,0.2){\line(0,1){0.9}}
     \put(2.3,0.2){\line(0,1){0.9}}
     \put(0.9,1.1){\line(1,0){1.4}}
     \put(0.05,-0.05){$^\uparrow$}
     \put(0.2,0.4){\tiny$^{\otimes s}$}
     \put(1.35,-0.05){$^\uparrow$}
     \put(1.5,0.4){\tiny$^{\otimes s}$}
     \put(0,-0.2){$\circ$}
     \put(0.7,-0.2){$\circ$}
     \put(1.3,-0.2){$\bullet$}
     \put(2.1,-0.2){$\bullet$}
     \end{picture}}
\newsavebox{\boxpositionerdinv} 
   \savebox{\boxpositionerdinv}
   { \begin{picture}(2.6,1.5)
     \put(0.9,0.2){\line(0,1){0.9}}
     \put(2.3,0.2){\line(0,1){0.9}}
     \put(0.9,1.1){\line(1,0){1.4}}
     \put(0.05,-0.05){$^\uparrow$}
     \put(0.2,0.4){\tiny$^{\otimes d}$}
     \put(1.35,-0.05){$^\uparrow$}
     \put(1.5,0.4){\tiny$^{\otimes d}$}
     \put(0,-0.2){$\circ$}
     \put(0.7,-0.2){$\bullet$}
     \put(1.3,-0.2){$\bullet$}
     \put(2.1,-0.2){$\circ$}
     \end{picture}}
\newsavebox{\boxpositionersinv} 
   \savebox{\boxpositionersinv}
   { \begin{picture}(2.6,1.5)
     \put(0.9,0.2){\line(0,1){0.9}}
     \put(2.3,0.2){\line(0,1){0.9}}
     \put(0.9,1.1){\line(1,0){1.4}}
     \put(0.05,-0.05){$^\uparrow$}
     \put(0.2,0.4){\tiny$^{\otimes s}$}
     \put(1.35,-0.05){$^\uparrow$}
     \put(1.5,0.4){\tiny$^{\otimes s}$}
     \put(0,-0.2){$\circ$}
     \put(0.7,-0.2){$\bullet$}
     \put(1.3,-0.2){$\bullet$}
     \put(2.1,-0.2){$\circ$}
     \end{picture}}
\newsavebox{\boxpositionerdpluszwei}
   \savebox{\boxpositionerdpluszwei}
   { \begin{picture}(3.4,1.5)
     \put(1.6,0.2){\line(0,1){0.9}}
     \put(3.0,0.2){\line(0,1){0.9}}
     \put(1.6,1.1){\line(1,0){1.4}}
     \put(0.05,-0.05){$^\uparrow$}
     \put(0.2,0.4){\tiny$^{\otimes d+2}$}
     \put(2.05,-0.05){$^\uparrow$}
     \put(2.2,0.4){\tiny$^{\otimes d}$}
     \put(0,-0.2){$\circ$}
     \put(1.4,-0.2){$\bullet$}
     \put(2,-0.2){$\bullet$}
     \put(2.8,-0.2){$\bullet$}
     \end{picture}}
\newsavebox{\boxpositionersminuszwei}
   \savebox{\boxpositionersminuszwei}
   { \begin{picture}(3.4,1.5)
     \put(1,0.2){\line(0,1){0.9}}
     \put(3.0,0.2){\line(0,1){0.9}}
     \put(1,1.1){\line(1,0){2}}
     \put(0.05,-0.05){$^\uparrow$}
     \put(0.2,0.4){\tiny$^{\otimes s}$}
     \put(1.45,-0.05){$^\uparrow$}
     \put(1.6,0.4){\tiny$^{\otimes s-2}$}
     \put(0,-0.2){$\circ$}
     \put(0.8,-0.2){$\bullet$}
     \put(1.4,-0.2){$\bullet$}
     \put(2.8,-0.2){$\bullet$}
     \end{picture}}
\newsavebox{\boxpositionerrnull}
   \savebox{\boxpositionerrnull}
   { \begin{picture}(3.8,1.5)
     \put(1.2,0.2){\line(0,1){0.9}}
     \put(3.4,0.2){\line(0,1){0.9}}
     \put(1.2,1.1){\line(1,0){2.2}}
     \put(0.05,-0.05){$^\uparrow$}
     \put(0.2,0.4){\tiny$^{\otimes r_0}$}
     \put(1.65,-0.05){$^\uparrow$}
     \put(1.8,0.4){\tiny$^{\otimes r_0-2}$}
     \put(0,-0.2){$\circ$}
     \put(1,-0.2){$\bullet$}
     \put(1.6,-0.2){$\bullet$}
     \put(3.2,-0.2){$\bullet$}
     \end{picture}}
\newsavebox{\boxpositionerwbwb}
   \savebox{\boxpositionerwbwb}
   { \begin{picture}(1.8,1)
     \put(0.6,0.2){\line(0,1){0.6}}
     \put(1.4,0.2){\line(0,1){0.6}}
     \put(0.6,0.8){\line(1,0){0.8}}
     \put(0.05,-0.05){$^\uparrow$}
     \put(0.85,-0.05){$^\uparrow$}
     \put(0,-0.2){$\circ$}
     \put(0.4,-0.2){$\bullet$}
     \put(0.8,-0.2){$\circ$}
     \put(1.2,-0.2){$\bullet$}
     \end{picture}}
\newsavebox{\boxpositionerwwbb}
   \savebox{\boxpositionerwwbb}
   { \begin{picture}(1.8,1)
     \put(0.6,0.2){\line(0,1){0.6}}
     \put(1.4,0.2){\line(0,1){0.6}}
     \put(0.6,0.8){\line(1,0){0.8}}
     \put(0.05,-0.05){$^\uparrow$}
     \put(0.85,-0.05){$^\uparrow$}
     \put(0,-0.2){$\circ$}
     \put(0.4,-0.2){$\circ$}
     \put(0.8,-0.2){$\bullet$}
     \put(1.2,-0.2){$\bullet$}
     \end{picture}}
\newsavebox{\boxpositionerwwww}
   \savebox{\boxpositionerwwww}
   { \begin{picture}(1.8,1)
     \put(0.6,0.2){\line(0,1){0.6}}
     \put(1.4,0.2){\line(0,1){0.6}}
     \put(0.6,0.8){\line(1,0){0.8}}
     \put(0.05,-0.05){$^\uparrow$}
     \put(0.85,-0.05){$^\uparrow$}
     \put(0,-0.2){$\circ$}
     \put(0.4,-0.2){$\circ$}
     \put(0.8,-0.2){$\circ$}
     \put(1.2,-0.2){$\circ$}
     \end{picture}}
\newsavebox{\boxpositionerrevwbwb}
   \savebox{\boxpositionerrevwbwb}
   { \begin{picture}(1.8,1)
     \put(0.2,0.2){\line(0,1){0.6}}
     \put(1.0,0.2){\line(0,1){0.6}}
     \put(0.2,0.8){\line(1,0){0.8}}
     \put(0.45,-0.05){$^\uparrow$}
     \put(1.25,-0.05){$^\uparrow$}
     \put(0,-0.2){$\circ$}
     \put(0.4,-0.2){$\bullet$}
     \put(0.8,-0.2){$\circ$}
     \put(1.2,-0.2){$\bullet$}
     \end{picture}}
\newsavebox{\boxpositionersalphaw} 
   \savebox{\boxpositionersalphaw}
   { \begin{picture}(2.6,1.5)
     \put(0.9,0.2){\line(0,1){0.9}}
     \put(2.3,0.2){\line(0,1){0.9}}
     \put(0.9,1.1){\line(1,0){1.4}}
     \put(0.05,-0.05){$^\uparrow$}
     \put(1.35,-0.05){$^\uparrow$}
     \put(0.2,0.4){\tiny$^{\otimes s}$}
     \put(1.5,0.4){\tiny$^{\otimes \alpha}$}
     \put(0,-0.2){$\circ$}
     \put(0.7,-0.2){$\circ$}
     \put(1.3,-0.2){$\circ$}
     \put(2.1,-0.2){$\bullet$}
     \end{picture}}
\newsavebox{\boxpositionersalphab} 
   \savebox{\boxpositionersalphab}
   { \begin{picture}(2.6,1.5)
     \put(0.9,0.2){\line(0,1){0.9}}
     \put(2.3,0.2){\line(0,1){0.9}}
     \put(0.9,1.1){\line(1,0){1.4}}
     \put(0.05,-0.05){$^\uparrow$}
     \put(1.35,-0.05){$^\uparrow$}
     \put(0.2,0.4){\tiny$^{\otimes s}$}
     \put(1.5,0.4){\tiny$^{\otimes \alpha}$}
     \put(0,-0.2){$\circ$}
     \put(0.7,-0.2){$\circ$}
     \put(1.3,-0.2){$\bullet$}
     \put(2.1,-0.2){$\bullet$}
     \end{picture}}

\newsavebox{\boxfatcrosswwww}
   \savebox{\boxfatcrosswwww}
   { \begin{picture}(2,1.6)
     \put(0.2,0.15){\line(0,1){0.2}}
     \put(0.2,0.95){\line(0,1){0.2}}
     \put(0.2,0.95){\line(1,0){0.5}}
     \put(0.2,0.35){\line(1,0){0.5}}
     \put(0.7,0.15){\line(0,1){0.2}}
     \put(0.7,0.95){\line(0,1){0.2}}
     \put(0,1.1){$\circ$}
     \put(0.5,1.1){$\circ$}
     \put(0,-0.2){$\circ$}
     \put(0.5,-0.2){$\circ$}
      \put(0.5,0.35){\line(2,1){1}}
      \put(0.5,0.85){\line(2,-1){1}}
     \put(1.2,0.15){\line(0,1){0.2}}
     \put(1.2,0.95){\line(0,1){0.2}}
     \put(1.2,0.95){\line(1,0){0.5}}
     \put(1.2,0.35){\line(1,0){0.5}}
     \put(1.7,0.15){\line(0,1){0.2}}
     \put(1.7,0.95){\line(0,1){0.2}}
     \put(1,1.1){$\circ$}
     \put(1.5,1.1){$\circ$}
     \put(1,-0.2){$\circ$}
     \put(1.5,-0.2){$\circ$}
     \end{picture}}
\newsavebox{\boxpairpositionerwwwwww}
   \savebox{\boxpairpositionerwwwwww}
   { \begin{picture}(2,1.6)
     \put(0.2,0.95){\line(0,1){0.2}}
     \put(0.2,0.95){\line(1,0){0.5}}
     \put(0.7,0.95){\line(0,1){0.2}}
     \put(0,1.1){$\circ$}
     \put(0.5,1.1){$\circ$}
     \put(0.25,-0.2){$\circ$}
       \put(0.45,0.1){\line(1,1){1}}
      \put(0.5,0.85){\line(2,-1){1}}
     \put(1.2,0.15){\line(0,1){0.2}}
     \put(1.2,0.35){\line(1,0){0.5}}
     \put(1.7,0.15){\line(0,1){0.2}}
     \put(1.25,1.1){$\circ$}
     \put(1,-0.2){$\circ$}
     \put(1.5,-0.2){$\circ$}
     \end{picture}}     

\newsavebox{\boxAspace}
   \savebox{\boxAspace}
   { \begin{picture}(0.1,1.5)
     \end{picture}}
\newsavebox{\boxBspace}
   \savebox{\boxBspace}
   { \begin{picture}(0,1.85)
     \end{picture}}

\newcommand{\idpartww}{\usebox{\boxidpartww}}

\newcommand{\idpartsingletonww}{\usebox{\boxidpartsingletonww}}

\newcommand{\paarpartww}{\usebox{\boxpaarpartww}}

\newcommand{\baarpartww}{\usebox{\boxbaarpartww}}

\newcommand{\paarbaarpartwwww}{\usebox{\boxpaarbaarpartwwww}}
\newcommand{\singletonw}{\usebox{\boxsingletonw}}

\newcommand{\downsingletonw}{\usebox{\boxdownsingletonw}}

\newcommand{\vierpartwwww}{\usebox{\boxvierpartwwww}}

\newcommand{\vierpartrotwwww}{\usebox{\boxvierpartrotwwww}}

\newcommand{\downvierpartwwww}{\usebox{\boxdownvierpartwwww}}
\newcommand{\dreipartwww}{\usebox{\boxdreipartwww}}

\newcommand{\crosspartonelinewwww}{\usebox{\boxcrosspartonelinewwww}}

\newcommand{\positionerwwww}{\usebox{\boxpositionerwwww}}

\newcommand{\paarpart}{\paarpartww}
\newcommand{\baarpart}{\baarpartww}
\newcommand{\paarbaarpart}{\paarbaarpartwwww}

\newcommand{\dreipart}{\dreipartwww}

\newcommand{\vierpart}{\vierpartwwww}
\newcommand{\downvierpart}{\downvierpartwwww}
\newcommand{\vierpartrot}{\vierpartrotwwww}

\newcommand{\crosspartoneline}{\crosspartonelinewwww}

\newcommand{\idpart}{\idpartww}
\newcommand{\singleton}{\singletonw}
\newcommand{\downsingleton}{\downsingletonw}

\newcommand{\positioner}{\positionerwwww}
\newcommand{\idpartsingleton}{\idpartsingletonww}

\usepackage{empheq}

\theoremstyle{plain} 
\newtheorem{thm}{Theorem}[section]
\newtheorem{prop}[thm]{Proposition}
\newtheorem{lem}[thm]{Lemma}
\newtheorem{cor}[thm]{Corollary}
\newtheorem*{thm*}{Theorem}
\theoremstyle{definition}
\newtheorem{defn}[thm]{Definition}
\newtheorem{rem}[thm]{Remark}
\newtheorem{ex}[thm]{Example}

\numberwithin{equation}{section}

\renewcommand{\theta}{\vartheta}
\renewcommand{\phi}{\varphi}
\renewcommand{\epsilon}{\varepsilon}
\renewcommand{\subset}{\subseteq}

\newcommand{\N}{\mathbb N}
\newcommand{\Z}{\mathbb Z}

\newcommand{\R}{\mathbb R}
\newcommand{\C}{\mathbb C}

\newcommand{\CC}{\mathcal C}

\DeclareMathOperator{\id}{id}
\DeclareMathOperator{\lspan}{span}
\DeclareMathOperator{\spec}{sp}
\DeclareMathOperator{\mult}{multinest}

\begin{document}
\title[Partition $C^*$-algebras II]{Partition $C^*$-algebras II -- Links to Compact Matrix Quantum Groups}
\author{Moritz Weber}
\address{Saarland University, Fachbereich Mathematik, Postfach 151150,
66041 Saarbr\"ucken, Germany}
\email{weber@math.uni-sb.de}
\date{\today}
\thanks{The author was supported by the ERC Advanced Grant NCDFP, held by Roland Speicher, by the SFB-TRR 195, and by the DFG project \emph{Quantenautomorphismen von Graphen}.}

\begin{abstract}
In a recent article, we gave a definition of partition $C^*$-algebras. These are universal $C^*$-algebras based on algebraic relations which are induced from partitions of sets.
In this follow up article, we show that often we can associate a Hopf algebra structure to partition $C^*$-algebras, and  also a  compact matrix quantum group structure. This follows the lines of Banica and Speicher's approach to quantum groups; however, we access them in a more algebraic way circumventing Tannaka-Krein duality. We give criteria when these quantum groups are quantum subgroups of Wang's free orthogonal quantum group. As a consequence, we see that even if we start with (generalized) categories of partitions which do not contain the pair partitions, in many cases we do not go beyond the class of Banica-Speicher quantum groups (aka easy quantum groups). However, we also discuss possible non-unitary Banica-Speicher quantum groups.
\end{abstract}

\maketitle

\section*{Introduction}

The definition of partition $C^*$-algebras \cite{PartI} is derived from Banica-Speicher quantum groups (also called easy quantum groups) \cite{BS,WeEQGLN,WeLInd}: Given a Banica-Speicher quantum group $(A,u)$, the $C^*$-algebra $A$ is a partition $C^*$-algebra. The intention for writing the article \cite{PartI} and the present follow up article was:

\begin{itemize}
\item[(1)] To provide an access to the purely $C^*$-algebraic side of Banica-Speicher quantum groups (forgetting the compact matrix quantum group structure) and to collect open problems for $C^*$-algebraists.
\item[(2)] To derive by algebraic means the definition of categories of partitions as well as the compact matrix quantum group structure of Banica-Speicher quantum groups; moreover, to prove algebraically that for defining $A$, we only need to impose the $C^*$-algebraic relations of the generators of a category rather than the relations of all elements in the category.
\item[(3)] To extend the Banica-Speicher machine to categories of partitions not containing the pair partitions $\paarpart,\baarpart$ in order to obtain non-orthogonal (or rather non-unitary) Banica-Speicher quantum groups.
\end{itemize}

While (1) and parts of (2) were successfully achieved in \cite{PartI}, the remainder of (2) is solved in Section \ref{SectExCMQGandHopf} of the present article. While working on (3) however, we discovered a phenomenon which we call ``enforced orthogonality'': Firstly, recall that the relations associated to the pair partitions $\paarpart$ and $\baarpart$ are equivalent to the fundamental matrix $u$ being orthogonal. Secondly, we define a \emph{generalized category of (non-colored) partitions} to be a set of partitions which is closed under taking the tensor product, the composition and the vertical reflection of partitions, and which contains the identy partition $\idpart$. If in addition, it is closed under taking the involution and if it contains the pair partitions $\paarpart$ and $\baarpart$, it is a category of partitions in Banica-Speicher's sense and the associated quantum group (obtained via Tannaka-Krein) is an orthogonal Banica-Speicher quantum group. Note that we do \emph{not} require the containment of $\paarpart$ and $\baarpart$ for a generalized category of partitions. On the other hand, we may prove in this article, that we still obtain a compact matrix quantum group in many cases.

Thus, our task (3) was: Starting with a generalized category of partitions not containing $\paarpart$ and $\baarpart$ -- do we obtain a compact matrix quantum group whose fundamental matrix $u$ is not orthogonal (nor unitary)? In other words, do we obtain a ``non-unitary Banica-Speicher quantum groups''  -- a compact matrix quantum group arising from partitions, but having the property that $u$ is not unitary? The answer is more difficult than expected, because of ``enforced orthogonality'': Even if we start with a generalized category of partitions, the matrix $u$ happens to be orthogonal in many cases. For instance, the relations associated to the partitions $\vierpart$ and $\downvierpart$ together imply orthogonality of $u$, too. In Section \ref{SectEnforcedOrth}, we study this enforced orthogonality in detail and in Section \ref{SectSummary} we discuss possible non-unitary Banica-Speicher quantum groups.

\section{Main results}

In  \cite{PartI}, we defined partition $C^*$-algebras as follows. 
Given a partition $p$ with $k$ upper and $l$ lower points and given $n\in\N$, we consider the following relations $R(p)$ for self-adjoint elements $u_{ij}$, with $i,j=1,\ldots,n$:
\[\boxed{R(p):\qquad\sum_{\gamma_1,\ldots,\gamma_k=1}^n \delta_p(\gamma,\beta) u_{\gamma_1\alpha_1}\ldots u_{\gamma_k\alpha_k}
=\sum_{\gamma_1',\ldots,\gamma_l'=1}^n \delta_p(\alpha,\gamma') u_{\beta_1\gamma_1'}\ldots u_{\beta_l\gamma_l'}\qquad\forall \alpha,\beta}\]
For a collection $X$ of partitions, we define a partition $C^*$-algebra as the following unital universal $C^*$-algebra:
\[\boxed{A_n(X):=C^*(1,u_{ij}, i,j=1,\ldots,n\;|\; u_{ij}=u_{ij}^*, \textnormal{the relations $R(p)$ hold for all }p\in X)}\]
See Section \ref{SectPart} and \cite{PartI} for details. Note that if $T_p$ is Banica and Speicher's linear map associated to a partition $p\in P(k,l)$, then the intertwiner relation $T_pu^{\otimes k}=u^{\otimes l}T_p$ is equivalent to $R(p)$.

While in \cite{PartI} we restricted to the $C^*$-algebraic side of partition $C^*$-algebras, we  now come to the quantum algebraic point of view. We  show that we may equip partition $C^*$-algebras with a quantum group structure. 

\begin{thm*}[Thm. \ref{ThmExCMQG1}]
Let $X\subset P$ be a set of partitions, let $p\in P(0,l)$ be a partition such that $p\neq \singleton^{\otimes l}$, and assume that the relations $R(p)$ and $R(p^*)$ hold in $A_n(X)$. Then $(A_n(X),(u_{ij}))$  is a compact matrix quantum group.
\end{thm*}

This is  a quite algebraic approach to Banica-Speicher type quantum groups circumventing the technicalities of Woronowicz's Tannaka-Krein theorem \cite{WoTK}. Moreover, while in Banica-Speicher's theory we always start with a category of partitions in order to abtain a quantum group (via Tannaka-Krein), we may start with arbitrary sets of partitions in our approach:

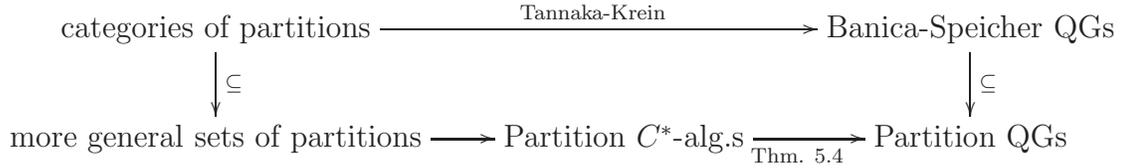
\begin{displaymath}
\xymatrix{
  \textnormal{categories of partitions} \ar[rr]^{\textnormal{Tannaka-Krein}} \ar[d]^{\subset}
 &&\textnormal{Banica-Speicher QGs}\ar[d]^{\subset}\\
 \textnormal{more general sets of partitions}
\ar[r]
 &\textnormal{Partition $C^*$-alg.s}\ar[r]_{\textnormal{Thm. \ref{ThmExCMQG1}}}
 &\textnormal{Partition QGs}
 }
\end{displaymath}

However, note that there is always a generalized category of partitions in the back, i.e. given an arbitrary set $X$, there is a generalized category of partitions $H_n(X)$ such that $A_n(X)=A_n(H_n(X))$. Moreover, even if $X$ does not contain the pair partitions $\paarpart$ and $\baarpart$, some partitions $p\in P$ are such that the relations $R(p)$ imply the relations $R(\paarpart)$ and $R(\baarpart)$ (i.e. orthogonality of $u$). The set of such partitions is denoted by $P_O\subset P$ and we have the following result on ``enforced orthogonality''.

\begin{thm*}[Thm. \ref{ThmCriteria}  and Prop. \ref{PropEnforcedOrth}]
Let $X\subset P$ and assume a few technical conditions on $X$. If $R(p)$ and $R(p^*)$ hold in $A_n(X)$ for some $p\in P_O$, then $A_n(X)=A_n(\CC_X)$, where $\CC_X$ is the Banica-Speicher category of partitions generated by $X$ (in particular $\paarpart,\baarpart\in\CC_X$). We investigate $P_O$ in detail.
\end{thm*}

So, partitions in $P_O$ bring us back to Banica-Speicher quantum groups. However, we expect that the partitions $\positioner$ and $\crosspartoneline$ are \emph{not} in $P_O$; if this was true, then this would lead to first examples of non-unitary Banica-Speicher quantum groups and potentially to a whole new class of compact matrix quantum groups, see Section \ref{SectSummary}. The missing part in our proof is to find concrete operators on a Hilbert space satisfying certain concrete relations, see \cite[Problems 3.5 \& 3.6]{PartI}.

\pagebreak

\section{Basics on partition $C^*$-algebras}
\label{SectPart}

In this section we recall some facts and notions from \cite{PartI}.

\subsection{Partitions and operations}
\label{SectRemind1}

Let $k,l\in\N_0$. A decomposition of the  ordered set $\{1,\ldots,k+l\}$ into disjoint, non-empty subsets (called the \emph{blocks}) whose union is the complete set, is called a \emph{(set) partition}. Such partitions are represented by pictures drawing lines between $k$ upper and $l$ lower points according to the block structure. As an example, the partition 
\[\{1,7,9\},\quad \{2,5\},\quad \{3\},\quad\{4\},\quad\{6,8\}\]
of the set $\{1,\ldots,9\}$, with  $k=4$ and $l=5$, may be drawn as:
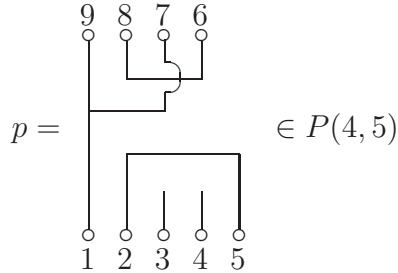
\newsavebox{\boxp}
   \savebox{\boxp}
   { \begin{picture}(4,5.5)
     \put(-1,6.35){\parti{5}{1}}
     \put(-1,6.35){\partii{1}{2}{4}}
     \put(0.3,3.5){\line(1,0){2}}
     \put(2.3,3.5){\line(0,1){0.5}}
     \put(2.3,4.8){\line(0,1){0.6}}
     \put(2.3,4.4){\oval(0.8,0.8)[r]}
     \put(0.05,5.3){$\circ$}
     \put(1.05,5.3){$\circ$}
     \put(2.05,5.3){$\circ$}
     \put(3.05,5.3){$\circ$}
     \put(-1,0.35){\uppartii{2}{2}{5}}
     \put(-1,0.35){\upparti{1}{3}}
     \put(-1,0.35){\upparti{1}{4}}     
     \put(0.05,0){$\circ$}
     \put(1.05,0){$\circ$}
     \put(2.05,0){$\circ$}
     \put(3.05,0){$\circ$}     
     \put(4.05,0){$\circ$}  
     \end{picture}}     
\begin{center}
\begin{picture}(10,7.5)
 \put(0,3.5){$p=$}
 \put(1.5,0.7){\usebox{\boxp}}
 \put(7,3.5){$\in P(4,5)$}
 \put(1.8,0){1}
 \put(2.8,0){2}
 \put(3.8,0){3}
 \put(4.8,0){4}
 \put(5.8,0){5}
 \put(1.8,6.5){9}
 \put(2.8,6.5){8}
 \put(3.8,6.5){7}
 \put(4.8,6.5){6}  
\end{picture}
\end{center}
We usually omit the numberings of the points. We denote by $P(k,l)$ the set of all partitions with $k$ upper and $l$ lower points, and we let $P$ be the union of all $P(k,l)$, $k,l\in\N_0$. 
Given partitions $p\in P(k,l)$ and $q\in P(k',l')$, we define
\begin{itemize}
\item the \emph{tensor product} $p\otimes q\in P(k+k',l+l')$ by horizontal concatenation, i.e. writing $p$ and $q$ side by side,
\item  the \emph{composition} $qp\in P(k,l')$ by vertical concatenation, i.e. writing $q$ below $p$ (here, we require $l=k'$),
\item the \emph{vertical reflection} $\tilde p\in P(k,l)$ by reflection at the vertical axis,
\item and the \emph{involution} $p^*\in P(l,k)$ by reflection at the horizontal axis, i.e. turning $p$ upside down.
\end{itemize}
See \cite{PartI} for examples of these operations.

\subsection{Partition relations and partition $C^*$-algebras}
\label{SectRemind2}

Let $n\in\N$ and let $A$ be a unital $C^*$-algebra generated by $n^2$ elements $u_{ij}$, $1\leq i,j\leq n$.
Let $p\in P(k,l)$ be a partition. We say that the generators $u_{ij}$ \emph{fulfill the relations} $R(p)$, if
all elements $u_{ij}$ are self-adjoint, and for all $\alpha_1,\ldots, \alpha_k\in\{1,\ldots,n\}$ and  all $\beta_1,\ldots,\beta_l\in\{1,\ldots,n\}$, we have:
\[\sum_{\gamma_1,\ldots,\gamma_k=1}^n \delta_p(\gamma,\beta) u_{\gamma_1\alpha_1}\ldots u_{\gamma_k\alpha_k}
=\sum_{\gamma_1',\ldots,\gamma_l'=1}^n \delta_p(\alpha,\gamma') u_{\beta_1\gamma_1'}\ldots u_{\beta_l\gamma_l'}
\]
In the cases $k=0$ or $l=0$ this is understood as:
\begin{align*}
\delta_p(\emptyset,\beta)1
&=\sum_{\gamma_1',\ldots,\gamma_l'=1}^n \delta_p(\emptyset,\gamma') u_{\beta_1\gamma_1'}\ldots u_{\beta_l\gamma_l'}
&\textnormal{if } k= 0, l\neq 0\\
\sum_{\gamma_1,\ldots,\gamma_k=1}^n \delta_p(\gamma,\emptyset) u_{\gamma_1\alpha_1}\ldots u_{\gamma_k\alpha_k}
&=\delta_p(\alpha,\emptyset)1
&\textnormal{if } k\neq 0, l= 0
\end{align*}
Here, $\delta_p(\alpha,\beta)\in\{0,1\}$ depending on whether or not the multi indices $\alpha$ and $\beta$ match the block structure of the partition $p$, see \cite[Sect. 2.1]{PartI}.

\begin{ex}
\label{ExOrth}
We have:
\begin{itemize}
\item[(a)] $R(\paarpart)$: $\sum_m u_{im}u_{jm}=\delta_{ij}$ for all $i,j$, which means $uu^t=1$ for $u=(u_{ij})$.
\item[(b)] $R(\baarpart)$: $\sum_m u_{mi}u_{mj}=\delta_{ij}$ for all $i,j$, which means $u^tu=1$.
\end{itemize}
\end{ex}

For $n\in\N$ and a set $X\subset P$, we define the \emph{partition $C^*$-algebra associated to} $X$ as  the unital universal $C^*$-algebra
\[A_n(X):=C^*(1, u_{ij}, i,j=1,\ldots,n\;|\; u_{ij}=u_{ij}^*, \textnormal{the relations $R(p)$ hold for all }p\in X),\]
if it exists. In this case, we say that $X$ is \emph{$n$-admissible}. One can show that $X$ is $n$-admissible for all $n\in\N$, if $\paarpart\in X$ or $\baarpart\in X$, see \cite[Lemma 2.4]{PartI}. See \cite[Sect. 2 and App. A]{PartI} for examples of relations $R(p)$ and $C^*$-algebras $A_n(X)$.
If $X\subset P$ is an $n$-admissible set of partitions, we denote by $H_n(X)$ the set of all partitions $p\in P$ such that the associated relations $R(p)$ hold in $A_n(X)$. We hence have $X\subset H_n(X)$; see also \cite[Lemma 2.7]{PartI} for more on $H_n(X)$.
The main theorem in \cite{PartI} is now that $H_n(X)$ is closed under several operations on partitions.

\begin{thm}[{\cite[Thm. 2.8]{PartI}}]\label{ThmPartI}
Let $X\subset P$ be an $n$-admissible set. Let $p\in P(k,l)$ and $q\in P(k',l')$.
\begin{itemize}
\item[(a)] We have $\idpart\in H_n(X)$.
\item[(b)] If $p,q\in H_n(X)$, then also $p\otimes q\in H_n(X)$.
\item[(c)] If $p,q\in H_n(X)$ and $l=k'$, then also $qp\in H_n(X)$.
\item[(d)] If $p\in H_n(X)$, then also $\tilde p\in H_n(X)$.
\item[(e)] Let  $\paarpart,\baarpart\in H_n(X)$. If $p\in H_n(X)$, then also $p^*\in H_n(X)$.
\end{itemize}
\end{thm}

See \cite[Sect. 4]{PartI} for extensions of the definition of partition $C^*$-algebras to other kinds of partitions, in particular to the non-selfadjoint situation $u_{ij}\neq u_{ij}^*$.

\pagebreak

\section{A combinatorial tool box}
\label{SectCombToolbox}

Before we pass to quantum algebraic structures associated to partition $C^*$-algebras, let us develop in this section certain tools and constructions which can be performed within $H_n(X)$ for an $n$-admissible set $X\subset P$. Since these tools are pureley combinatorial, we may formulate it for slightly more general sets, namely for generalized categories of partitions.

\subsection{Categories of partitions in two senses}

\begin{defn}
\label{DefCateg}
Let $\CC\subset P$ be the union of subsets $\CC(k,l)\subseteq P(k,l)$, for all $k,l\in\N_0$. 
\begin{itemize}
\item[(a)] $\CC$ is called a \emph{generalized category of partitions}, if $p,q\in\CC$ implies $p\otimes q\in\CC$, $pq\in\CC$ and $\tilde p\in\CC$ (for the composition, we assume that the number of lower points of $q$ matches the number of upper points of $p$). Moreover, we require $\idpart\in\CC$. For partitions $p_1,\ldots,p_m\in P$, we denote by $\langle p_1,\ldots,p_m\rangle_g$ the smallest generalized category of partitions containing $p_1,\ldots,p_m$; we say that this category is \emph{generated} by $p_1,\ldots,p_m$.
\item[(b)] $\CC$ is called a \emph{(Banica-Speicher) category of partitions}, if in addition $p\in\CC$ implies $p^*\in\CC$ and we have $\paarpart,\baarpart\in\CC$.
\end{itemize}
\end{defn}

Categories of partitions in the sense of Definition \ref{DefCateg}(b) were introduced first by Banica and Speicher \cite{BS}, see also \cite{WeEQGLN, WeLInd}. Note that in this case, the vertical reflection may be derived from the other operations, \cite[Lemma 1.1]{TWcomb}. See Section \ref{SectExNonBS} for an example of a generalized category of partitions which is not a Banica-Speicher category of partitions.
The key to the generalized Definition \ref{DefCateg}(a) is the following rephrasing of Theorem \ref{ThmPartI}.

\begin{thm}[{\cite[Thm. 2.8]{PartI}}]
Let $X\subset P$ be an $n$-admissible set. Then $H_n(X)$ is a generalized category of partitions. If moreover $\paarpart,\baarpart\in H_n(X)$, then $H_n(X)$ is even a Banica-Speicher category of partitions.
\end{thm}

The remainder of this section is formulated for generalized categories of partitions, having $H_n(X)$ in mind.

\subsection{Joining and cutting blocks. Nestings}
\label{SectNest}

We recall some basics on generalized categories of partitions from \cite[App. B]{PartI} or rather \cite{TWcomb}.

\begin{lem}
\label{LemBasicOper}
Let $\CC\subset P$ be a generalized category of partitions.
\begin{itemize}
\item[(a)]  let $p\in P(0,l)$ and $q\in P(0,m)$. If $p,q\in \CC$, then every partition obtained from placing $q$ between two legs of $p$ is in $\CC$. We say that $q$ is \emph{nested} between two legs of $p$.
\item[(b)] If $\idpartsingletonww\in H_n(X)$, then $\CC$ is closed under  disconnecting any point from a block of a partition $p\in\CC$ and turning it into a singleton.
\item[(c)]  If $\vierpartrot\in\CC$, then $\CC$ is closed under  connecting neighbouring blocks of a partition $p\in\CC$ on the same line.
\end{itemize}
\end{lem}
\begin{proof}
The assertions follow from composing $p$ with  $\idpart^{\alpha}\otimes q\otimes\idpart^{\beta}$ for suitable $\alpha$ and $\beta$ in case (a), see \cite[Lem. 1.1(d)]{TWcomb}; with  $\idpart^{\alpha}\otimes \idpartsingletonww\otimes\idpart^{\beta}$ in case (b), see \cite[Lem. 1.3]{TWcomb}, and  $\idpart^{\alpha}\otimes \vierpartrot\otimes\idpart^{\beta}$ in case (c), see  \cite[Lem. 1.3]{TWcomb}.
\end{proof}

For $p\in P(0,l)$,  $1\leq s\leq l$ and $m\in\N$, we denote by 
\[\mult(p,s,m)\in P(0,lm)\]
the partition obtained from \emph{multi nesting} $p$ between the $s$-th and $(s+1)$-th legs of itself, $m$ times, i.e.
\begin{itemize}
\item we nest one copy of $p$ between the  $s$-th and $(s+1)$-th legs of $p$ in the sense of Lemma \ref{LemBasicOper}(a), obtaining $\mult(p,s,2)$;
\item we then  recursively nest  $\mult(p,s,r)$ between the  $s$-th and $(s+1)$-th legs of $p$, for $2\leq r\leq m-1$;
\item and we finally obtain $\mult(p,s,m)$.
\end{itemize}

Here is an example of a multi nesting:

\begin{center}
\begin{picture}(20,5)
 \put(0,0){$\mult(\dreipart,1,4)=$}
 \put(7.5,0){\uppartiii{4}{1}{11}{12}}
 \put(7.5,0){\uppartiii{3}{2}{9}{10}}
 \put(7.5,0){\uppartiii{2}{3}{7}{8}}
 \put(7.5,0){\uppartiii{1}{4}{5}{6}}  
\end{picture}
\end{center}

\subsection{Line rotation}

Let $p,p'\in P(0,l)$ be two partitions. 
\begin{itemize}
\item[(a)] If for all $2\leq i<j\leq l$, the points $i$ and $j$ are in the same block of $p$ if and only if the points $i-1$ and $j-1$ are in the same block of $p'$, we say that $p'$ is a \emph{weakly line rotated version of} $p$.
\item[(b)] If in addition for all $2\leq j\leq l$ the points $1$ and $j$ are in the same block of $p$ if and only if the points $j-1$ and $l$ are in the same block of $p'$, we say that $p'$ is a \emph{line rotated version of} $p$.
\end{itemize}
The above defines a (weak) line rotation of the first point of $p$ to the right of all of its other points. After a line rotation, this last point is connected to the block to which it was connected before the rotation; for a weak line rotation, this requirement is dropped.
 We may also define a (weak) line rotation to the left. In the sequel, we will not specify left or right rotation.
Here is an example of a weakly line rotated version $p'$ of $p$ and a line rotated version $p''$ of $p$.

\begin{center}
\begin{picture}(29,4)
 \put(0,1){$p=$}
 \put(1,1){\uppartiv{2}{1}{4}{6}{7}}
 \put(1,1){\upparti{1}{2}}
 \put(1,1){\uppartii{1}{3}{5}}
 \put(10,1){$p'=$}
 \put(11,1){\uppartiii{2}{3}{5}{6}}
 \put(11,1){\upparti{1}{1}}
 \put(11,1){\uppartii{1}{2}{4}} 
 \put(11,1){\upparti{1}{7}}
 \put(11,0){(weakly line rotated)}
 \put(20,1){$p''=$}
 \put(21,1){\uppartiv{2}{3}{5}{6}{7}}
 \put(21,1){\upparti{1}{1}}
 \put(21,1){\uppartii{1}{2}{4}} 
 \put(23,0){(line rotated)} 
\end{picture}
\end{center}

\begin{lem}
\label{LemWeakRot}
Let $\CC\subset P$ be a generalized category of partitions, let $p\in P(0,l)$ and $q\in P(k,0)$. Assume $p\in\CC$.
\begin{itemize}
\item[(a)] If $q\in\CC$, then any weakly line rotated version of $p$ is in $\CC$.
\item[(b)] If $q,q^*\in\CC$ and $q^*\neq\singleton^{\otimes k}$, then any line rotated version of $p$ is in $\CC$.
\end{itemize}
\end{lem}
\begin{proof}
(a) The following partition $p'$ is a weakly line rotated version of $p$.
\newsavebox{\boxweaki}
   \savebox{\boxweaki}
   { \begin{picture}(9,3)
      \put(4.25,1.5){$p^{\otimes k}$}
      \put(0,1){\line(0,1){1.5}}
      \put(9.5,1){\line(0,1){1.5}}
      \put(0,1){\line(1,0){9.5}}    
      \put(0,2.5){\line(1,0){9.5}}          
      \put(-1,0){\upparti{1}{1}}
      \put(1.2,0){$\cdots$}
      \put(-1,0){\upparti{1}{4}}      
      \put(-1,0){\upparti{1}{10}}      
     \end{picture}}     
\newsavebox{\boxweakii}
   \savebox{\boxweakii}
   { \begin{picture}(4,3)
      \put(2,1.5){$p$}      
      \put(0,1){\line(0,1){1.5}}
      \put(4.5,1){\line(0,1){1.5}}
      \put(0,1){\line(1,0){4.5}}    
      \put(0,2.5){\line(1,0){4.5}}          
      \put(-1,0){\upparti{1}{1}}
      \put(-1,0){\upparti{1}{2}}      
      \put(2.2,0){$\cdots$}
      \put(-1,0){\upparti{1}{5}}      
     \end{picture}}     
\newsavebox{\boxweakiii}
   \savebox{\boxweakiii}
   { \begin{picture}(4,3)
      \put(2,1){$q^{\otimes l}$}         
      \put(0,0.5){\line(0,1){1.5}}
      \put(4.5,0.5){\line(0,1){1.5}}
      \put(0,0.5){\line(1,0){4.5}}    
      \put(0,2){\line(1,0){4.5}}          
      \put(-1,2){\upparti{1}{1}}
      \put(1.2,2.5){$\cdots$}
      \put(-1,2){\upparti{1}{4}}      
      \put(-1,2){\upparti{1}{5}}      
     \end{picture}}     
\begin{center}
\begin{picture}(28,10)
 \put(0,3){$p':=\left(q^{\otimes l}\otimes\idpart^{\otimes l}\right)\left(\idpart^{\otimes (kl-1)}\otimes p\otimes\idpart\right)p^{\otimes k}=$}
 \put(17.7,7){\usebox{\boxweaki}}
 \put(18.05,6.65){$\circ$}
 \put(21.05,6.65){$\circ$} 
 \put(27.05,6.65){$\circ$} 
 \put(17,3.7){\upparti{3}{1}}
 \put(19.2,5){$\cdots$}
 \put(17,3.7){\upparti{3}{4}}
 \put(21.7,3.7){\usebox{\boxweakii}}
 \put(17,3.7){\upparti{3}{10}}
 \put(18.05,3.35){$\circ$}
 \put(21.05,3.35){$\circ$} 
 \put(22.05,3.35){$\circ$} 
 \put(23.05,3.35){$\circ$}   
 \put(26.05,3.35){$\circ$} 
 \put(27.05,3.35){$\circ$} 
 \put(17.7,0.4){\usebox{\boxweakiii}}
 \put(17,0.4){\upparti{3}{6}}
 \put(24.2,1.5){$\cdots$}
 \put(17,0.4){\upparti{3}{9}} 
 \put(17,0.4){\upparti{3}{10}} 
 \put(23.05,0.05){$\circ$}   
 \put(26.05,0.05){$\circ$} 
 \put(27.05,0.05){$\circ$} 
\end{picture}
\end{center}
It is in $\CC$, since $\CC$ is closed under tensor products and compositions. 

(b) Since $q^*\neq\singleton^{\otimes k}$, there are two points $1\leq s<t\leq k$ which are in the same block of $q^*$. By weak line rotation, we may assume $t=k$. Replacing $p^{\otimes k}$ by $q^*$ and $q^{\otimes l}$ by $q$ in the construction of $p'$ in (a), we infer that the partition
\[p':=\left(q\otimes\idpart^{\otimes l}\right)\left(\idpart^{\otimes (k-1)}\otimes p\otimes\idpart\right)q^*\]
is a line rotated version of $p$, which is in $\CC$. Note that as the last point of $q$ is connected to some point $s$, which in turn is connected by $q^*$ to the very right point of the construction, we ensure that the first point of $p$ is not disconnected from the other points of $p$ throughout the procedure.
\end{proof}

Note that if $\CC$ is a Banica-Speicher category of partitions, the partition $q=\paarpart$ does the job in Lemma \ref{LemWeakRot}(b). See also \cite[App. B]{PartI} for more on rotations of partitions in Banica-Speicher categories.

\subsection{Doubling}
\label{SectDoubling}

Given a partition $p\in P(0,l)$ and a number $1\leq s\leq l$, we call the partition
\newsavebox{\boxdoubli}
   \savebox{\boxdoubli}
   { \begin{picture}(7,3)
      \put(3.5,1.5){$p$}
      \put(-0.2,0.2){1}
      \put(2.8,0.2){$s$}      
      \put(0,1){\line(0,1){1.5}}
      \put(7.5,1){\line(0,1){1.5}}
      \put(0,1){\line(1,0){7.5}}    
      \put(0,2.5){\line(1,0){7.5}}          
      \put(-1,0){\upparti{1}{1}}
      \put(1.2,0){$\cdots$}
      \put(-1,0){\upparti{1}{4}}      
      \put(-1,0){\upparti{1}{5}}  
      \put(5.2,0){$\cdots$}      
      \put(-1,0){\upparti{1}{8}}      
     \end{picture}}     
\newsavebox{\boxdoublii}
   \savebox{\boxdoublii}
   { \begin{picture}(7,3)
      \put(3.5,1){$\tilde p^*$}  
      \put(4.4,2.2){$s$}       
      \put(7.4,2.2){1}             
      \put(0,0.5){\line(0,1){1.5}}
      \put(7.5,0.5){\line(0,1){1.5}}
      \put(0,0.5){\line(1,0){7.5}}    
      \put(0,2){\line(1,0){7.5}}          
      \put(-1,2){\upparti{1}{1}}
      \put(1.2,2.5){$\cdots$}
      \put(-1,2){\upparti{1}{4}}      
      \put(-1,2){\upparti{1}{5}} 
      \put(5.2,2.5){$\cdots$}           
      \put(-1,2){\upparti{1}{8}} 
     \end{picture}}     
\begin{center}
\begin{picture}(24,7)
 \put(0,3){$q:=\left(\idpart^{\otimes s}\otimes \tilde p^*\right)\left(p\otimes\idpart^{\otimes s}\right)=$}
 \put(12.7,4){\usebox{\boxdoubli}}
 \put(13.05,3.6){$\circ$}
 \put(16.05,3.6){$\circ$} 
 \put(17.05,3.6){$\circ$} 
 \put(20.05,3.6){$\circ$}  
 \put(21.05,3.6){$\circ$}  
 \put(24.05,3.6){$\circ$}  
 \put(12,3.95){\upparti{3}{9}} 
 \put(12,3.95){\upparti{3}{12}}  
 \put(14.2,2){$\cdots$}
 \put(21.05,6.8){$\circ$}  
 \put(24.05,6.8){$\circ$}  
 \put(12,0.7){\upparti{3}{1}}
 \put(12,0.7){\upparti{3}{4}} 
 \put(16.7,0.7){\usebox{\boxdoublii}}
 \put(13.05,0.35){$\circ$}
 \put(16.05,0.35){$\circ$} 
\end{picture}
\end{center}

the \emph{shifted doubling (on $s$ legs) of} $p$.

\begin{lem}\label{LemWeakDoubling}
Let $\CC\subset P$ be a generalized category of partitions and let $p\in P(0,l)$. If $p,p^*\in \CC$, then all shifted doublings of $p$ are in $\CC$, for all $1\leq s\leq l$.
\end{lem}
\begin{proof}
Clear from the construction. Note that $\widetilde{(p^*)}=(\tilde p)^*$.
\end{proof}

The following construction of $q\in P(s,s)$  is called the \emph{partial doubling of $p\in P(0,l)$ (on its first $s$ legs)}.
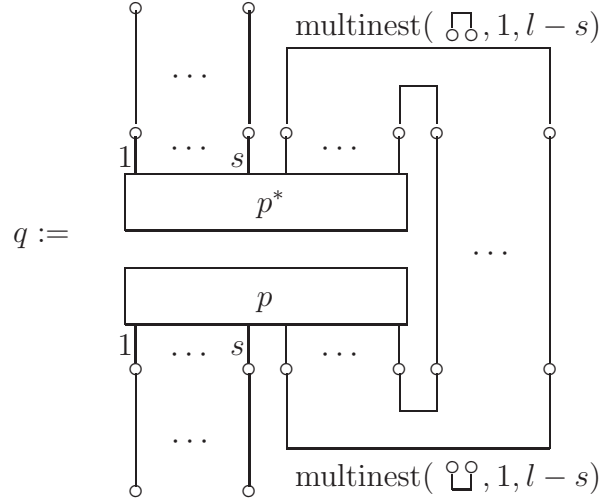
\newsavebox{\boxpartdoublii}
   \savebox{\boxpartdoublii}
   { \begin{picture}(7,3)
      \put(3.5,1){$p^*$}  
      \put(-0.2,2.2){1}             
      \put(2.8,2.2){$s$}       
      \put(0,0.5){\line(0,1){1.5}}
      \put(7.5,0.5){\line(0,1){1.5}}
      \put(0,0.5){\line(1,0){7.5}}    
      \put(0,2){\line(1,0){7.5}}          
      \put(-1,2){\upparti{1}{1}}
      \put(1.2,2.5){$\cdots$}
      \put(-1,2){\upparti{1}{4}}      
      \put(-1,2){\upparti{1}{5}} 
      \put(5.2,2.5){$\cdots$}           
      \put(-1,2){\upparti{1}{8}} 
     \end{picture}}     
\begin{center}
\begin{picture}(16,14)
 \put(0,7){$q:=$}
 \put(2,10.15){\upparti{3}{1}}
 \put(2,10.15){\upparti{3}{4}} 
 \put(2,10.15){\uppartii{2}{5}{12}} 
 \put(2,10.15){\uppartii{1}{8}{9}}  
 \put(2.7,6.8){\usebox{\boxpartdoublii}}
 \put(2.7,3.8){\usebox{\boxdoubli}} 
 \put(2,3.8){\upparti{6}{9}}
 \put(2,3.8){\upparti{6}{12}}
 \put(2,4.5){\parti{3}{1}}
 \put(2,4.5){\parti{3}{4}} 
 \put(2,4.5){\partii{2}{5}{12}} 
 \put(2,4.5){\partii{1}{8}{9}}  
 \put(3.05,0.15){$\circ$}
 \put(6.05,0.15){$\circ$} 
 \put(3.05,3.4){$\circ$} 
 \put(6.05,3.4){$\circ$} 
 \put(7.05,3.4){$\circ$} 
 \put(10.05,3.4){$\circ$} 
 \put(11.05,3.4){$\circ$} 
 \put(14.05,3.4){$\circ$} 
 \put(3.05,9.7){$\circ$} 
 \put(6.05,9.7){$\circ$} 
 \put(7.05,9.7){$\circ$} 
 \put(10.05,9.7){$\circ$} 
 \put(11.05,9.7){$\circ$} 
 \put(14.05,9.7){$\circ$}  
 \put(3.05,13){$\circ$}
 \put(6.05,13){$\circ$} 
 \put(4.2,1.5){$\cdots$}
 \put(4.2,11.2){$\cdots$}
 \put(12.2,6.5){$\cdots$} 
 \put(7.5,0.5){${\mult(\baarpart,1,l-s)}$}
 \put(7.5,12.5){${\mult(\paarpart,1,l-s)}$}
\end{picture}
\end{center}

The crucial point is, that if a point $1\leq r\leq s$ of $p$ is connected to some point $s<t\leq l$, then the lower $r$-th point of $q$ is connected to the upper $r$-the point of $q$.

\begin{lem}\label{LemPartDoubling}
Let $\CC\subset P$ be a generalized category of partitions and let $p\in P(0,l)$ with $p\neq \singleton^{\otimes l}$. If $p,p^*\in\CC$, then all partial doublings of $p$ are in $\CC$, for all $1\leq s\leq l$.
\end{lem}
\begin{proof}
By weak line rotation (Lemma \ref{LemWeakRot}), we obtain a partition $p'\in P(0,l)$ whose first point is not a singleton. The proof is then clear from the partial doubling construction with $\mult(p',1,l-s)$ playing the role of $\mult(\paarpart,1,l-s)$, while $\mult(p'^*,1,l-s)$ plays the role of $\mult(\baarpart,1,l-s)$.
\end{proof}

\begin{lem}\label{LemPartDoublingProj}
Let $q\in P(s,s)$ be a partition obtained from partial doubling. Then $q=q^*=qq$.
\end{lem}
\begin{proof}
Clear from the construction. 
\end{proof}

\begin{lem}\label{LemProjPart}
Let $q\in P(s,s)$ be a partition such that $q=q^*=qq$. Then, if an upper point $i$ and a lower point $j$ are in the same block $V$, then also the lower point $i$ and the upper point $j$ are in $V$.
\end{lem}
\begin{proof}
Using $q=q^*$, check that if an upper point $i$ is connected to a lower point $j$, then also the upper point $j$ is connected to the lower point $i$. By $q=qq$, this yields that the upper point $i$ is connected to the lower point $i$, while the upper point $j$ is connected to the lower point $j$. This yields that all four points (upper $i$ and $j$ as well as lower $i$ and $j$) are in the same block. See also \cite{FW} for such an argument.
\end{proof}

The partial doubling procedure  may be extended a bit: Let $p_1\in P(0,l)$ and $p_2\in P(k,0)$ and let $q\in P(s,s)$. Let $1\leq a< b\leq s$. If $\CC\subset P$ is a generalized category of partitions and if $p_1,p_2,q\in \CC$, then the following \emph{weak restriction of} $q$ is in $\CC\cap P(b-a,b-a)$:

\[\left(p_2^{\otimes l}\otimes \idpart^{\otimes (b-a)}\otimes p_2^{\otimes l}\right)\left(\idpart^{\otimes(kl-a)}\otimes q\otimes\idpart^{\otimes(kl-(s-b))}\right)\left(p_1^{\otimes k}\otimes\idpart^{\otimes (b-a)}\otimes p_1^{\otimes k}\right)\]
\newsavebox{\boxwresti}
   \savebox{\boxwresti}
   { \begin{picture}(7,3)
      \put(3.5,1.5){$p_1^{\otimes k}$}
      \put(0,1){\line(0,1){1.5}}
      \put(7.5,1){\line(0,1){1.5}}
      \put(0,1){\line(1,0){7.5}}    
      \put(0,2.5){\line(1,0){7.5}}          
      \put(-1,0){\upparti{1}{1}}
      \put(1.2,0){$\cdots$}
      \put(-1,0){\upparti{1}{4}}      
      \put(-1,0){\upparti{1}{5}}  
      \put(5.2,0){$\cdots$}      
      \put(-1,0){\upparti{1}{8}}      
     \end{picture}}  
\newsavebox{\boxwrestii}
   \savebox{\boxwrestii}
   { \begin{picture}(7,3)
      \put(3.5,1){$p_2^{\otimes l}$}  
      \put(0,0.5){\line(0,1){1.5}}
      \put(7.5,0.5){\line(0,1){1.5}}
      \put(0,0.5){\line(1,0){7.5}}    
      \put(0,2){\line(1,0){7.5}}          
      \put(-1,2){\upparti{1}{1}}
      \put(1.2,2.5){$\cdots$}
      \put(-1,2){\upparti{1}{4}}      
      \put(-1,2){\upparti{1}{5}} 
      \put(5.2,2.5){$\cdots$}           
      \put(-1,2){\upparti{1}{8}} 
     \end{picture}}     
\newsavebox{\boxwrestiii}
   \savebox{\boxwrestiii}
   { \begin{picture}(12,3)
      \put(5.5,1.3){$q$}  
      \put(0,1){\line(0,1){1}}
      \put(11.5,1){\line(0,1){1}}
      \put(0,1){\line(1,0){11.5}}    
      \put(0,2){\line(1,0){11.5}}          
      \put(-1,2){\upparti{1}{1}}
      \put(-1,2){\upparti{1}{4}}      
      \put(-1,2){\upparti{1}{5}} 
      \put(-1,2){\upparti{1}{8}} 
      \put(-1,2){\upparti{1}{9}} 
      \put(-1,2){\upparti{1}{12}} 
      \put(-1,0){\upparti{1}{1}}
      \put(-1,0){\upparti{1}{4}}      
      \put(-1,0){\upparti{1}{5}} 
      \put(-1,0){\upparti{1}{8}} 
      \put(-1,0){\upparti{1}{9}} 
      \put(-1,0){\upparti{1}{12}} 
      \put(0.4,2.2){$_1$}
      \put(3.4,2.2){$_a$}
      \put(4.4,2.2){$_{a+1}$}
      \put(7.4,2.2){$_b$}
      \put(8.4,2.2){$_{b+1}$}
      \put(11.4,2.2){$_s$}                              
     \end{picture}}        
\begin{center}
\begin{picture}(22,10)
 \put(0,5){$=$}
 \put(10.35,9.85){$\circ$}
 \put(13.35,9.85){$\circ$}
 \put(2,7){\usebox{\boxwresti}}
 \put(1.25,7){\upparti{3}{9}}
 \put(11.5,8.5){$\cdots$}
 \put(1.25,7){\upparti{3}{12}}
 \put(14,7){\usebox{\boxwresti}}
 \put(2.35,6.65){$\circ$}
 \put(5.35,6.65){$\circ$}
 \put(6.35,6.65){$\circ$}
 \put(9.35,6.65){$\circ$}
 \put(10.35,6.65){$\circ$}
 \put(13.35,6.65){$\circ$}
 \put(14.35,6.65){$\circ$}
 \put(17.35,6.65){$\circ$}
 \put(18.35,6.65){$\circ$}
 \put(21.35,6.65){$\circ$}
 \put(1.25,3.7){\upparti{3}{1}}
 \put(3.5,5){$\cdots$}
 \put(1.25,3.7){\upparti{3}{4}}
 \put(6,3.7){\usebox{\boxwrestiii}}
 \put(1.25,3.7){\upparti{3}{17}}
 \put(20.2,5){$\cdots$} 
 \put(1.25,3.7){\upparti{3}{20}} 
 \put(2.35,3.35){$\circ$}
 \put(5.35,3.35){$\circ$}
 \put(6.35,3.35){$\circ$}
 \put(9.35,3.35){$\circ$}
 \put(10.35,3.35){$\circ$}
 \put(13.35,3.35){$\circ$}
 \put(14.35,3.35){$\circ$}
 \put(17.35,3.35){$\circ$}
 \put(18.35,3.35){$\circ$}
 \put(21.35,3.35){$\circ$}   
 \put(2,0.4){\usebox{\boxwrestii}}
 \put(1.25,0.4){\upparti{3}{9}}
 \put(10.8,1.9){$\cdots$}
 \put(1.25,0.4){\upparti{3}{12}} 
 \put(14,0.4){\usebox{\boxwrestii}} 
 \put(10.35,0.05){$\circ$}
 \put(13.35,0.05){$\circ$}
\end{picture}
\end{center}

\subsection{An example of a generalized category of partitions}
\label{SectExNonBS}

Let $m\in\N$. Let $NC_{[m]}\subset P$ be the set of all noncrossing partitions $p$ satisfying the following \emph{grading condition} on the block structure: If a block of $p$ contains $a$ upper and $b$ lower points, then $a-b\in m\Z$. For $m\geq 3$, these sets are categories in the sense of Definition \ref{DefCateg}(a) but not in the sense of Definition \ref{DefCateg}(b), see Proposition \ref{PropExNonBSCateg}. By $b_m\in P(0,m)$ we denote the partition consisting in a single block on $m$ lower points:
\[b_1=\singleton,\qquad b_2=\paarpart,\qquad b_3=\dreipart,\qquad b_4=\vierpart,\ldots\]
In order to prove Proposition \ref{PropExNonBSCateg}, we need the following lemma about a kind of an ``$m$-rotation'' of the upper left point to the lower line.

\begin{lem}
\label{LemExNonBSCateg}
Let $m\in\N$ and let $p\in P(k,l)$, $k\geq 1$. If $p\in NC_{[m]}$, then also $p'\in NC_{[m]}$, where $p'\in P(k-1,l+(m-1))$ is defined as:
%
\newsavebox{\boxmroti}
   \savebox{\boxmroti}
   { \begin{picture}(5,3)
      \put(2,1.5){$b_m$}
      \put(-1,0){\uppartiii{1}{1}{4}{5}}
      \put(1.2,0){$\cdots$}
     \end{picture}}     
\newsavebox{\boxmrotii}
   \savebox{\boxmrotii}
   { \begin{picture}(7,3)
      \put(2.5,1.3){$p$}  
      \put(0,1){\line(0,1){1}}
      \put(5.5,1){\line(0,1){1}}
      \put(0,1){\line(1,0){5.5}}    
      \put(0,2){\line(1,0){5.5}}          
      \put(-1,2){\upparti{1}{1}}
      \put(2.7,2.5){$\cdots$}
      \put(-1,2){\upparti{1}{2}}      
      \put(-1,2){\upparti{1}{6}} 
      \put(-1,0){\upparti{1}{1}}
      \put(2.2,0.2){$\cdots$}
      \put(-1,0){\upparti{1}{2}}      
      \put(-1,0){\upparti{1}{5}} 
     \end{picture}}     
\begin{center}
\begin{picture}(24,7)
 \put(0,3){$p':=\left(\idpart^{\otimes m-1} \otimes p\right)\left(b_m\otimes \idpart^{\otimes k-1}\right)=$}
 \put(14.7,4){\usebox{\boxmroti}}
 \put(15.05,3.6){$\circ$}
 \put(18.05,3.6){$\circ$} 
 \put(19.05,3.6){$\circ$} 
 \put(20.05,3.6){$\circ$}  
 \put(24.05,3.6){$\circ$}  
 \put(14,3.95){\upparti{3}{6}} 
 \put(14,3.95){\upparti{3}{10}}  
 \put(16.2,2){$\cdots$}
 \put(20.05,6.8){$\circ$}  
 \put(24.05,6.8){$\circ$}  
 \put(14,0.7){\upparti{3}{1}}
 \put(14,0.7){\upparti{3}{4}} 
 \put(18.7,0.7){\usebox{\boxmrotii}}
 \put(15.05,0.35){$\circ$}
 \put(18.05,0.35){$\circ$} 
 \put(19.05,0.35){$\circ$} 
 \put(20.05,0.35){$\circ$}  
 \put(23.05,0.35){$\circ$}  
\end{picture}
\end{center}
\end{lem}
\begin{proof}
Let $V_1,\ldots, V_l$ be the blocks of $p$ and suppose that $V_1$ contains the upper left point of $p$. Then the blocks of $p'$ are $V_1',V_2,\ldots,V_l$, where $V_1'$ contains the lower left point of $p'$. In fact, $p$ and $p'$ only differ by the blocks $V_1$ and $V_1'$ respectively. Now, suppose that $V_1$ has $a$ upper and $b$ lower points. Then, $V_1'$ has $a-1$ upper and $(m-1)+b$ lower points, and we have
$(a-1)-((m-1)+b)=a-b-m\in m\Z$.
\end{proof}

We may now show that the sets $NC_{[m]}$ are generalized categories of partitions but no Banica-Speicher categories for $m\geq 3$. However, we will see later that the sets $H_n(NC_{[m]})$ \emph{are} Banica-Speicher categories of partitions, see Remark \ref{RemNCm}.

\begin{prop}
\label{PropExNonBSCateg}
The sets $NC_{[m]}$ are generalized categories of partitions for all $m\in\N$; in fact, they are even closed under taking involutions. However, the sets $NC_{[m]}$ are no Banica-Speicher categories for $m\geq 3$. Their generators are:
\[NC_{[1]}=\langle \paarpart,\baarpart,\dreipart\rangle_g,\quad NC_{[2]}=\langle\paarpart,\baarpart,\vierpart\rangle_g,\quad NC_{[m]}=\langle b_m, b_m^*\rangle_g, m\geq 3\]
\end{prop}
\begin{proof}
Let $m\in \N$. Firstly, note that $b_m,b_m^*\in NC_{[m]}$ for all $m\in\N$. Secondly, it is straighforward to check that $NC_{[m]}$ is closed under taking  tensor products, vertical reflections and even involutions of partitions; moreover, $\idpart\in NC_{[m]}$. Thus we are left with proving that $NC_{[m]}$ is closed under compositions in order to see that it is a generalized category of partitions. We prove it by induction with respect to the number $n$ of middle points in the composition. Let $p\in P(n,l)$ and $q\in P(k,n)$. Assume $p,q\in NC_{[m]}$.

\emph{Case 1.} If $n=0$, then $pq\in P(k,l)$ is the partition obtained from placing the partition $q\in P(k,0)$ on the upper points and $p\in P(0,l)$ on the lower points. Any block of $pq$ is either a block of $p$ or a block of $q$ and hence the grading condition of $NC_{[m]}$ is satisfied; hence $pq\in NC_{[m]}$.

\emph{Case 2.} If $n=1$, there is exactly one block in $q$ containing $a$ upper points and one lower point, such that $a-1\in m\Z$. Likewise, $p$ contains exactly one block on $b$ lower points and one upper point, such that $b-1=-(1-b)\in m\Z$. Now, composing $p$ and $q$, we obtain a partition $pq\in P(k,l)$ containing one block on $a$ upper and $b$ lower points. Now, $a-b=(a-1)-(b-1)\in m\Z$. Moreover, all other blocks of $pq$ either consist only in upper points or only in lower points. Thus, taking into account Case 1, we infer $pq\in NC_{[m]}$.

\emph{Case 3.} If $n\geq 2$, we consider the partitions:
\[p':=\left(\idpart^{\otimes m-1} \otimes p\right)\left(b_m\otimes \idpart^{\otimes n-1}\right),\qquad
q':=\left(b_m^*\otimes \idpart^{\otimes n-1}\right)\left(\idpart^{\otimes m-1} \otimes q\right)\]
Then $p'\in P(n-1,l+(m-1))$ and $q'\in P(k+(m-1),n-1)$ are in $NC_{[m]}$ by Lemma \ref{LemExNonBSCateg}. By the induction hypothesis, the composition $p'q'$ is in $NC_{[m]}$. Moreover:
\[pq=\left(b_m^*\otimes\idpart^{\otimes l}\right)\left(\idpart\otimes p'q'\right)\left(b_m\otimes \idpart^{\otimes k}\right)\]
Using Lemma \ref{LemExNonBSCateg}, this completes the proof that $NC_{[m]}$ is a generalized category of partitions.

For $m\geq 3$, it is clear, that $\paarpart\notin NC_{[m]}$, so $NC_{[m]}$ is no Banica-Speicher category of partitions. Regarding the generators, all we have to prove is $NC_{[m]}\subset \langle b_m,b_m^*\rangle_g$ for $m\geq 3$. Note that $NC_{[1]}=NC$ while $NC_{[2]}$ consists in all noncrossing partitions having blocks of even size, so for $m=1$ and $m=2$, the situation is clear from the known theory on Banica-Speicher categories of partitions, see for instance \cite{We}.

As for proving $NC_{[m]}\subset \langle b_m,b_m^*\rangle_g$, $m\geq 3$, we first convince ourselves that $b_{nm}, b_{nm}^*\in \langle b_m,b_m^*\rangle_g$ for all $n\in\N$. Indeed, using the multinest notation from Section \ref{SectNest}, we have:
\begin{align*}
b_{2m}&=\left(b_m^*\otimes \idpart^{\otimes 2m}\right)\mult(b_m,1,3)\in \langle b_m,b_m^*\rangle_g\\
b_{(n+1)m}&=\left(\idpart^{\otimes (m+1)}\otimes b_m^*\otimes \idpart^{\otimes (nm-1)}\right)\left(b_{2m}\otimes b_{nm}\right)\in \langle b_m,b_m^*\rangle_g, \quad n\geq 2
\end{align*}
Now, let $p\in NC_{[m]}$ be a partition consisting in a single block, with  $a$ upper and $b$ lower points. Let $x,y\in\N$ be such that $xm>a\geq(x-1)m$ and $ym>b\geq(y-1)m$. As $a-b\in m\Z$, we have $b-(y-1)m=a-(x-1)m$. We infer:
\[p=\left(b_m^*\otimes\idpart^{\otimes b}\right)\left(\idpart^{\otimes (b-(y-1)m)}\otimes b_{ym}\otimes b_{xm}^*\right)\left(b_m\otimes\idpart^{\otimes a}\right)\in \langle b_m,b_m^*\rangle_g\]
Thus, $p\in\langle b_m,b_m^*\rangle_g$. Using nestings and tensor products, we infer $NC_{[m]}\subset\langle b_m,b_m^*\rangle_g$.
\end{proof}

\section{Basics on Banica-Speicher quantum groups and others}

In this section we recall the definitions of Hopf algebras (Section \ref{SectHopf}), compact matrix quantum groups (Section \ref{SectCMQG}) and Banica-Speicher quantum groups (Section \ref{SectBSQG}) building on Woronowicz's Tannaka-Krein duality (Section \ref{SectTK}). In Section \ref{SectOrignPart} we explain how partition $C^*$-algebras originate from the theory of Banica-Speicher quantum groups.

\subsection{Hopf algebras}
\label{SectHopf}

Let $A$ be a unital algebra. We denote by $m:A\otimes A\to A, (a,b)\mapsto ab$ the multiplication map 
and by $\eta:\C\to A,\lambda\mapsto \lambda 1$ the embedding of the complex numbers as multiples of the identity.

A \emph{(unital) Hopf algebra} is a tuple $(A,\Delta,\epsilon, S)$ such that
\begin{itemize}
\item $A$ is a unital complex  algebra,
\item $\Delta:A\to A\otimes A$ is a unital algebra homomorphism (the \emph{comultiplication}) with \[(\Delta\otimes\id)\circ\Delta=(\id\otimes\Delta)\circ\Delta,\]
\item $\epsilon: A\to \C$ is an algebra homomorphism (the \emph{counit}) with 
\[(\epsilon\otimes \id)\circ\Delta=\id=(\id\otimes\epsilon)\circ\Delta,\]
\item $S:A\to A$ is a linear map (the \emph{antipode}) with 
\[m\circ(S\otimes\id)\circ\Delta=\eta\circ\epsilon=m\circ(\id\otimes S)\circ\Delta.\]
\end{itemize}

See also \cite{Tim} for more on Hopf algebras.

\subsection{Compact matrix quantum groups}
\label{SectCMQG}

Let $n\in \N$. A \emph{compact matrix quantum group} is a tuple $G=(A,u)$ such that
\begin{itemize}
\item $A$ is a unital $C^*$-algebra generated by elements $u_{ij}$, for $i,j=1,\ldots, n$,
\item the matrices $u=(u_{ij})$ and $\bar u=(u_{ij}^*)$ are invertible in $M_n(A)$,
\item $\Delta: A\to A\otimes_{\min} A, u_{ij}\mapsto \sum_k u_{ik}\otimes u_{kj}$ is a unital $^*$-homomorphism.
\end{itemize}
We sometimes denote $C(G):=A$, even when $A$ is noncommutative.

Compact matrix quantum groups have been introduced by Woronowicz \cite{WoCMQG, WoCMQG2}, see also \cite{Tim, NT}. Some authors prefer to replace the condition that $\bar u$ needs to be invertible by imposing invertibility of $u^t=(u_{ji})$. It is straightforward to check that both conditions are equivalent. 
Hopf algebras and compact matrix quantum groups are naturally linked: Any compact matrix quantum group contains a dense Hopf algebra, see \cite[Sect. 5.4]{Tim} or \cite[Sect. 1.6]{NT}.

\subsection{Tannaka-Krein duality for compact matrix quantum groups}
\label{SectTK}

 Woronowicz \cite{WoTK} proved a quantum version of Tannaka-Krein duality, also called Schur-Weyl duality. See also \cite{Tim, NT, Mala}.

\begin{prop}[Tannaka-Krein duality]
\label{PropTK}
Compact matrix quantum groups are completely determined by their representation theory. More precisely:
\begin{itemize}
\item[(a)] If $G$ is a compact matrix quantum group, then the class of finite-dimensional unitary representations $\textnormal{Rep}(G)$ forms a tensor category in Woronowicz's sense.
\item[(b)] Conversely, if $R$ is an (abstract) tensor category in Woronowicz's sense, there is a compact matrix quantum group $G$ such that $\textnormal{Rep}(G)=R$.
\end{itemize}
\end{prop}

The key step in the above duality between tensor categories and compact matrix quantum groups is the interpretation of the morphisms from the tensor category as intertwiners of the quantum group. An intertwiner is a linear map $T:H_1\to H_2$ between finite-dimensional Hilbert spaces such that $Tu_1=u_2T$ holds for representations $u_1$ and $u_2$ of the quantum group.

\subsection{Banica-Speicher quantum groups}
\label{SectBSQG}

Due to the work of Banica and Speicher \cite{BS} we may associate a Woronowicz tensor category to a category of partitions: 
Let $n\in\N$ and let $e_1,\ldots,e_n$ be the canonical basis of $\C^n$. Given a partition $p\in P(k,l)$, we define a linear map $T_p:(\C^n)^{\otimes k}\to(\C^n)^{\otimes l}$ by:
\[T_p(e_{i_1}\otimes \ldots\otimes e_{i_k})=\sum_{j_1,\ldots, j_l=1}^n\delta_p(i,j)e_{j_1}\otimes \ldots\otimes e_{j_l}\]
If $k=0$, then $T_p:\C\to(\C^n)^{\otimes l}$ is given by $T_p(1)=\sum_{j_1,\ldots, j_l=1}^n\delta_p(\emptyset,j)e_{j_1}\otimes \ldots\otimes e_{j_l}$. Analogously, if $l=0$ then $T_p:(\C^n)^{\otimes k}\to\C$ is defined by $T_p(e_{i_1}\otimes \ldots\otimes e_{i_k})=\delta_p(i,\emptyset)$.
One can check \cite[Prop. 1.9]{BS} that the tensor product, the composition and the involution on partitions passes through  the assignment $p\mapsto T_p$ to the tensor product, composition and involution of the maps $T_p$. Hence, given a category of partitions $\CC$, the set
\[\lspan\{T_p\;|\; p\in\CC\}\]
is a (weak version of a) tensor category in Woronowicz's sense, and we may associate a compact matrix quantum group $(A,u)$ to it, a \emph{Banica-Speicher quantum group}, by definition. As sketched above, the maps $T_p$ are then interpreted as intertwiners of representations of $(A,u)$. See \cite{BS} or also \cite{WeEQGLN, TWApp, WeLInd}.

\subsection{Extracting partition $C^*$-algebras from the Banica-Speicher theory}
\label{SectOrignPart}

The definition of partition $C^*$-algebras is very much inspired from the Banica-Speicher theory, as will be explained in this section.
Let $(A,u)$ be any compact matrix quantum group. There are canonical representations of $(A,u)$ given by tensor powers $u^{\otimes k}\in M_{n^k\times n^k}(\C)\otimes A$ of the matrix $u$, for $k\in \N$. They  act on elements $e_{i_1}\otimes\ldots\otimes e_{i_k}\otimes 1$ of $(\C^n)^k\otimes A$ in the following way.  We write:
\[u^{\otimes k}(e_{i_1}\otimes\ldots\otimes e_{i_k}\otimes 1)=\sum_{j_1,\ldots,j_k=1}^n e_{j_1}\otimes\ldots\otimes e_{j_k}\otimes u_{j_1 i_1}\ldots u_{j_ki_k}\]
Given a partition $p\in P(k,l)$, the map $T_p$ is an intertwiner for $u^{\otimes k}$ and $u^{\otimes l}$ if and only if the relations $R(p)$ hold in $A$.

\begin{lem}\label{LemIntertwiner}
 Let $p\in P(k,l)$ and let $n\in\N$. Let $A$ be a $C^*$-algebra generated by elements $u_{ij}$, $1\leq i,j\leq n$. Then, the generators $u_{ij}$ fulfill the relations $R(p)$ if and only if $T_p u^{\otimes k}=u^{\otimes l} T_p$ and all $u_{ij}$ are self-adjoint.
\end{lem}
\begin{proof}
For any multi index $\alpha=(\alpha_1,\ldots,\alpha_k)$ we have:
\begin{align*}
T_p u^{\otimes k} (e_{\alpha_1}\otimes\ldots\otimes e_{\alpha_k}\otimes 1)
&= \sum_{\gamma_1,\ldots,\gamma_k} T_p (e_{\gamma_1}\otimes\ldots\otimes e_{\gamma_k})\otimes u_{\gamma_1\alpha_1}\ldots u_{\gamma_k\alpha_k}\\
&= \sum_{\gamma_1,\ldots,\gamma_k} \sum_{\beta_1,\ldots,\beta_l} \delta_p(\gamma,\beta)e_{\beta_1}\otimes\ldots\otimes e_{\beta_l}\otimes u_{\gamma_1\alpha_1}\ldots u_{\gamma_k\alpha_k}\\
&=  \sum_{\beta_1,\ldots,\beta_l}e_{\beta_1}\otimes\ldots\otimes e_{\beta_l}\otimes \left(\sum_{\gamma_1,\ldots,\gamma_k}  \delta_p(\gamma,\beta) u_{\gamma_1\alpha_1}\ldots u_{\gamma_k\alpha_k}\right)
\end{align*}
On the other hand:
\begin{align*}
u^{\otimes l}T_p  (e_{\alpha_1}\otimes\ldots\otimes e_{\alpha_k}\otimes 1)
&= \sum_{\gamma_1',\ldots,\gamma_l'} u^{\otimes l} \delta_p(\alpha,\gamma')(e_{\gamma_1'}\otimes\ldots\otimes e_{\gamma_l'}\otimes 1)\\
&= \sum_{\gamma_1',\ldots,\gamma_l'} \sum_{\beta_1,\ldots,\beta_l} \delta_p(\alpha,\gamma') e_{\beta_1}\otimes\ldots\otimes e_{\beta_l}\otimes u_{\beta_1\gamma_1'}\ldots u_{\beta_l\gamma_l'}\\
&= \sum_{\beta_1,\ldots,\beta_l} e_{\beta_1}\otimes\ldots\otimes e_{\beta_l}\otimes \left(\sum_{\gamma_1',\ldots,\gamma_l'}\delta_p(\alpha,\gamma')u_{\beta_1\gamma_1'}\ldots u_{\beta_l\gamma_l'}\right)
\end{align*}
Comparison of the coefficients yields the result.
\end{proof}

\begin{prop}
\label{PropBSisPart}
Let $n\in\N$ and let $\CC\subset P$ be a category of partitions. Let $(A,u)$ be the Banica-Speicher quantum group arising via Tannaka-Krein duality. Then $A$ is a partition $C^*$-algebra:
\[A=A_n(\CC)\]
\end{prop}
\begin{proof}
Following the lines of the Tannaka-Krein Theorem (Prop. \ref{PropTK}), we infer that $A$ is defined as the universal $C^*$-algebra generated by the intertwiner relations $T_pu^{\otimes k}=u^{\otimes l}T_p$; see \cite{TWop} or rather \cite{TWApp} for more on this, also \cite{WeEQGLN, WeLInd}. By Lemma \ref{LemIntertwiner} we infer $A=A_n(\CC)$.
\end{proof}

In Section \ref{SectExCMQG} we will show how to circumvent the Tannaka-Krein theorem in the construction of Banica-Speicher quantum groups. In this context, we also want to recall a result from \cite{PartI}.

\begin{prop}[{\cite[Cor. 2.9]{PartI}}]
Let $\mathcal C=\langle p_1,\ldots,p_m\rangle$ be a category of partitions generated by partitions $p_1,\ldots,p_m$ (i.e. $\mathcal C$ is the smallest category of partitions containing $p_1,\ldots,p_m$), then:
\[A_n(\mathcal C)=A_n(\{p_1,\ldots,p_m,\paarpart,\baarpart\})\]
\end{prop}

\section{Quantum algebraic structures on partition $C^*$-algebras}
\label{SectExCMQGandHopf}

In this section, we show how partition $C^*$-algebras may be equiped with quantum algebraic structures, namely with the structure of a Hopf algebra on the one hand, and a compact matrix quantum group on the other hand. This is by no means surprising: The quantum algebraic structures are just the same as the standard structures on Wang's $O_n^+$ and $S_n^+$ \cite{WaOn, WaSn} or the Banica-Speicher quantum groups. However, we give purely algebraic proofs and we work with a class of $C^*$-algebras which is larger than the one underlying Banica-Speicher quantum groups.

\subsection{Existence of a Hopf algebra structure}

\begin{prop}\label{PropExHopf}
Let $X\subset P$ be an $n$-admissible set.
\begin{itemize}
\item[(a)] There is a unital $^*$-homomorphism $\Delta:A_n(X)\to A_n(X)\otimes A_n(X)$, mapping $u_{ij}\mapsto \sum_k u_{ik}\otimes u_{kj}$ with $(\Delta\otimes\id)\circ\Delta=(\id\otimes\Delta)\circ\Delta$.
\item[(b)] There is a unital $^*$-homomorphism $\epsilon: A_n(X)\to\C$, mapping $u_{ij}\to\delta_{ij}$ with $(\epsilon\otimes \id)\circ\Delta=\id=(\id\otimes\epsilon)\circ\Delta$.
\item[(c)] If $H_n(X)$ is closed under taking adjoints of partitions, then there is a linear map $S: A_n(X)\to A_n(X)$, mapping $u_{ij}\to u_{ji}$. If moreover $\paarpart,\baarpart\in H_n(X)$, then $m\circ(S\otimes\id)\circ\Delta=\eta\circ\epsilon=m\circ(\id\otimes S)\circ\Delta$ holds.
\end{itemize}
\end{prop}
\begin{proof}
Throughout the proof, let $p\in X$, let $\alpha$ be a multi index of length $k$, and let $\beta$ be a multi index of length $l$.

(a) Let $u_{ij}':=\sum_k u_{ik}\otimes u_{kj}$. We have to show that the elements $u_{ij}'$ fulfill the relations $R(p)$. By the universal property of $A_n(X)$, we may then conclude the existence of $\Delta$.
We compute, using twice the relations $R(p)$ for the $u_{ij}$:
\begin{align*}
\sum_\gamma \delta_p(\gamma,\beta)u'_{\gamma_1\alpha_1}\ldots u'_{\gamma_k\alpha_k}
&= \sum_\zeta\left(\sum_\gamma \delta_p(\gamma,\beta)u_{\gamma_1\zeta_1}\ldots u_{\gamma_k\zeta_k}\right)\otimes u_{\zeta_1\alpha_1}\ldots u_{\zeta_k\alpha_k}\\
&= \sum_\zeta\left(\sum_{\zeta'} \delta_p(\zeta,\zeta')u_{\beta_1\zeta_1'}\ldots u_{\beta_l\zeta_l'}\right)\otimes u_{\zeta_1\alpha_1}\ldots u_{\zeta_k\alpha_k}\\
&= \sum_{\zeta'}u_{\beta_1\zeta_1'}\ldots u_{\beta_l\zeta_l'}\otimes\left(\sum_\zeta \delta_p(\zeta,\zeta') u_{\zeta_1\alpha_1}\ldots u_{\zeta_k\alpha_k}\right)\\
&= \sum_{\zeta'}u_{\beta_1\zeta_1'}\ldots u_{\beta_l\zeta_l'}\otimes\left(\sum_{\gamma'} \delta_p(\alpha,\gamma') u_{\zeta_1'\gamma_1'}\ldots u_{\zeta_l'\gamma_l'}\right)\\
&=\sum_{\gamma'} \delta_p(\alpha,\gamma')\left( \sum_{\zeta'}u_{\beta_1\zeta_1'}\ldots u_{\beta_l\zeta_l'}\otimes u_{\zeta_1'\gamma_1'}\ldots u_{\zeta_l'\gamma_l'}\right)\\
&=\sum_{\gamma'} \delta_p(\alpha,\gamma')u'_{\beta_1\gamma_1'}\ldots u'_{\beta_l\gamma_l'}
\end{align*}
It is then routine to check $(\Delta\otimes\id)\circ\Delta=(\id\otimes\Delta)\circ\Delta$.

(b) Put $u'_{ij}:=\delta_{ij}$ and check:
\[\sum_\gamma \delta_p(\gamma,\beta)u'_{\gamma_1\alpha_1}\ldots u'_{\gamma_k\alpha_k}
=\delta_p(\alpha,\beta)
=\sum_{\gamma'} \delta_p(\alpha,\gamma')u'_{\beta_1\gamma_1'}\ldots u'_{\beta_l\gamma_l'}
\]
Again, $(\epsilon\otimes \id)\circ\Delta=\id=(\id\otimes\epsilon)\circ\Delta$ is straightforward to check.

(c) For $u_{ij}':=u_{ji}$, the relations 
\[\sum_{\gamma_1,\ldots,\gamma_k=1}^n \delta_p(\gamma,\beta) u'_{\gamma_1\alpha_1}\ldots u'_{\gamma_k\alpha_k}
=\sum_{\gamma_1',\ldots,\gamma_l'=1}^n \delta_p(\alpha,\gamma') u'_{\beta_1\gamma_1'}\ldots u'_{\beta_l\gamma_l'}\]
are equivalent to $R(p^*)$. Moreover $m\circ(S\otimes\id)\circ\Delta(u_{ij})=\sum_k u_{ki}u_{kj}$. If now $\paarpart,\baarpart\in H_n(X)$, then $\sum_k u_{ki}u_{kj}=\eta\circ\epsilon(u_{ij})$. Similarly $\eta\circ\epsilon=m\circ(\id\otimes S)\circ\Delta$.
\end{proof}

\begin{cor}
If $X\subset P$ is $n$-admissible  and if $\paarpart, \baarpart\in H_n(X)$, then $A_n(X)$ may be endowed with a Hopf algebra structure.
\end{cor}

\subsection{Existence of a compact matrix quantum group structure}
\label{SectExCMQG}

Since the generators $u_{ij}\in A_n(X)$ are self-adjoint and since the existence of a map $\Delta$ is ensured by  Proposition \ref{PropExHopf}(a) (even in the case $\paarpart,\baarpart\notin X$), the only thing left to prove is the invertibility of the matrix $u$, if we want to equip $A_n(X)$ with a compact matrix quantum group structure.

\begin{lem}\label{LemInvert}
Let $X\subset P$ be $n$-admissible.
\begin{itemize}
\item[(a)] Let  $p\in P(0,l)$ be such that the first point of $p$ is not a singleton. If $p\in H_n(X)$, then $u$ is right invertible.
\item[(b)] Let $q\in P(k,0)$ be such that the last point of $q$ is not a singleton.  If $q\in H_n(X)$, then $u$ is left invertible.
\end{itemize}
\end{lem}
\begin{proof}
(a) The first point of $p$ is connected to the $r$-th point, for some $1< r\leq l$. Choose indices 
\[\beta_2,\ldots,\beta_{r-1},\beta_{r+1},\ldots,\beta_l\in\{1,\ldots,n\}\]
 such that for all $i,j\in\{1,\ldots,n\}$:
\[\delta_p(\emptyset,(i,\beta_2,\ldots,\beta_{r-1},j,\beta_{r+1},\ldots,\beta_l))=\delta_{ij}\]
Put:
\[t_{ij}:=\sum_{\gamma_2',\ldots,\gamma_l'}\delta_p(\emptyset,(i,\gamma_2',\ldots,\gamma_l'))u_{\beta_2\gamma_2'}\ldots u_{\beta_{r-1}\gamma'_{r-1}}u_{j\gamma'_r}u_{\beta_{r+1}\gamma'_{r+1}}\ldots u_{\beta_l\gamma_l'}\]
Then $\sum_{\gamma_1'} u_{i\gamma_1'}t_{\gamma_1'j}=\delta_{ij}$ by relation $R(p)$, hence $(t_{ij})$ is a right inverse of $u$.

(b) The proof for (b) is similar.
\end{proof}

As an immediate consequence, we obtain:

\begin{thm}\label{ThmExCMQG1}
Let $X\subset P$ be $n$-admissible. Let $p\in P(0,l)$ with $p\neq \singleton^{\otimes l}$ and $q\in P(k,0)$ with $q\neq \downsingleton^{\otimes k}$ for some $k,l\geq 2$. If $p,q\in H_n(X)$, then $(A_n(X),u)$ is a compact matrix quantum group. In particular, if $X$ is a Banica-Speicher category of partitions, $(A_n(X),u)$ is a compact matrix quantum group.
\end{thm}
\begin{proof}
By weak rotation (Lemma \ref{LemWeakRot}) and since $p\neq\singleton^{\otimes l}$, we may assume that the first point of $p$ is not a singleton; likewise for the last point of $q$. Thus, by Lemma \ref{LemInvert} the matrix $u$ is invertible and using Proposition \ref{PropExHopf}(a), we obtain the result.
\end{proof}

We conclude that there are two ways of obtaining a compact matrix quantum group given a category of partitions $\CC$: Either one takes the road of Tannaka-Krein (Sections \ref{SectTK}, \ref{SectBSQG}), which is the original construction of Banica and Speicher; or one uses our direct algebraic approach.

\section{Enforced orthogonality}
\label{SectEnforcedOrth}

Coming back to one of the main motivations to write this article, let us summarize what we have so far, building on Theorem \ref{ThmExCMQG1} -- how far may we go beyond Banica-Speicher quantum groups?
\begin{itemize}
\item[(1)] Let $\CC\subset P$ be a Banica-Speicher category of partitions, i.e. $p,q\in\CC$ implies $p\otimes q,pq,\tilde p,p^*\in\CC$. Moreover, $\paarpart,\baarpart,\idpart\in\CC$.

Then $A_n(\CC)$ is a compact matrix quantum group. 

In fact it is a Banica-Speicher quantum group.
\item[(2)] Let $\CC\subset P$ be a generalized category of partitions,  i.e. $p,q\in\CC$ implies $p\otimes q,pq,\tilde p\in\CC$ and we have $\idpart\in\CC$. We do not necessarily have $\paarpart,\baarpart\in\CC$, but we assume the much weaker condition that $\CC$ contains partitions $p\in P(0,l)$ with $p\neq \singleton^{\otimes l}$ and $q\in P(0,k)$ with $q\neq \downsingleton^{\otimes k}$.

Then $A_n(\CC)$ is a compact matrix quantum group.

Is it a Banica-Speicher quantum group?
\end{itemize}

A priori, there are many examples of generalized categories of partitions which are not Banica-Speicher categories of partitions,  but which fit into Case (2) above. Consider for instance the categories $NC_{[m]}$ from Proposition \ref{PropExNonBSCateg} -- are the associated compact matrix quantum groups also Banica-Speicher quantum groups? The answer is yes as will be shown in this section. The reason is, that certain relations $R(p)$ for $p\in H_n(X)$ imply $\paarpart,\baarpart\in H_n(X)$ -- and the generator $b_m$ of $NC_{[m]}$ is one of them; thus $A_n(X)=A_n(\CC_X)$ by \cite[Lem. 2.7(e)]{PartI}, where $\CC_X$ is the smallest Banica-Speicher category of partitions containing $X$. However, we will also discuss situations $A_n(X)\neq A_n(\CC_X)$ in Section \ref{SectSummary}.

\subsection{The notion of enforced orthogonality}

\begin{defn}
A partition $p\in P$ is said to \emph{enforce orthogonality}, if 
\begin{itemize}
\item for all $n\in\N$ and all $n$-admissible sets $X\subset P$
\item with $H_n(X)\cap P(0,l)\neq\emptyset$  and $H_n(X)\cap P(k,0)\neq\emptyset$ for some $k,l\in\N$
\item and  $p,p^*\in H_n(X)$,
\item we have $\paarpart,\baarpart\in H_n(X)$.
\end{itemize}
 The set $P_O\subset P$ is by definition the set of all partitions $p\in P$ enforcing orthogonality.
\end{defn}

The following proposition explains the notion of the previous definition as well as our interest in the set $P_O$ -- or rather in its complement.

\begin{prop}\label{PropEnforcedOrth}
Let $p\in P_O$. Let $X\subset P$ be an $n$-admissible set with $H_n(X)\cap P(0,l)\neq\emptyset$  and $H_n(X)\cap P(k,0)\neq\emptyset$ for some $k,l\in\N$.  If $p,p^*\in H_n(X)$, then $H_n(X)$ is a Banica-Speicher category of partitions, $(A_n(X),u)$ is a Banica-Speicher quantum group and the matrix $u$ is orthogonal.
\end{prop}
\begin{proof}
From $p,p^*\in H_n(X)$, we infer $\paarpart,\baarpart\in H_n(X)$ by the definition of $P_O$. From Theorem \ref{ThmPartI}, we know that $H_n(X)$ is a Banica-Speicher category of partitions. From Proposition \ref{PropBSisPart}, we know that the Banica-Speicher quantum group associated to $H_n(X)$ is given by $(A_n(H_n(X),u)$, with $A_n(X)=A_n(H_n(X))$ (see also \cite[Lem. 2.7(c)]{PartI}). By Example \ref{ExOrth}, we know that $u$ is orthogonal.
\end{proof}

Enforced orthogonality it is the key feature of partition $C^*$-algebras to be studied when seeking new examples of ``partition''  quantum groups going beyond Banica-Speicher quantum groups. We will use our combinatorial toolbox (Setion \ref{SectCombToolbox}) as well as the following operator algebraic toolbox in order to prove that certain partitions $p_1,\ldots,p_m\in H_n(X)$ imply the existence of other partitions $q_1,\ldots,q_k\in H_n(X)$

\subsection{An operator algebraic toolbox: Selfadjoint idempotents in $C^*$-algebras}

The following innocent lemma on selfadjoint idempotents in $C^*$-algebras is the crucial ingredient for deriving the existence of partitions $q\in H_n(X)$, if we want to go beyond the (purely combinatorial) partition calculus of Section \ref{SectCombToolbox}.

\begin{lem}[Idempotents Lemma]
\label{LemIdem}
Let $A$ be a unital $C^*$-algebra and let $z\in A$ be such that $z=z^*$.
\begin{itemize}
\item[(a)] If $z^k=1$ for some $k\in\N$, then $z^2=1$.
\item[(b)] If $z^k=1$ for some odd number $k\in\N$, then $z=1$.
\item[(c)] If $z$ is positive and $z^k=1$ for some $k\in\N$, then $z=1$.
\end{itemize} 
\end{lem} 
\begin{proof}
Since $z$ is selfadjoint, the $C^*$-subalgebra $C^*(z,1)\subset A$ is commutative. It is therefore isomorphic to the algebra $C(\spec(z))$ of continuous functions on the spectrum $\spec(z)\subset \R$ of $z$. The isomorphism maps $z$ to the identity function $\id$ on $\spec(z)$. From $\id^k=1$ we infer $\spec(z)\subset \{-1,1\}$ if $k$ is even and $\spec(z)=\{1\}$ if $k$ is odd or if $z$ is positive. Thus, $\id^2=1$ in the first case and $\id=1$ in the latter.
\end{proof}

\subsection{Blocks of length greater or equal to three}

The next lemma makes use of the operator algebraic toolbox.

\begin{lem}
\label{LemPisigma}
For $m\geq 2$, the following partitions from $P(m,m)$  are in $P_O$:
\newsavebox{\boxpim}
   \savebox{\boxpim}
   
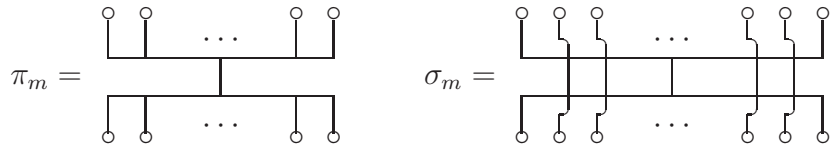
{ \begin{picture}(7,3)
      \put(1.1,3){$\circ$}
      \put(2.1,3){$\circ$}
      \put(6.1,3){$\circ$}
      \put(7.1,3){$\circ$}                  
      \put(0,0){\uppartiv{1}{1}{2}{6}{7}}
      \put(3.8,2.3){$\cdots$}
      \put(0,1){\upparti{1}{4}}
      \put(0,4){\partiv{1}{1}{2}{6}{7}}      
      \put(3.8,0){$\cdots$}
      \put(1.1,-0.3){$\circ$}
      \put(2.1,-0.3){$\circ$}
      \put(6.1,-0.3){$\circ$}
      \put(7.1,-0.3){$\circ$}    
     \end{picture}}  
\newsavebox{\boxsigmambogen}
   \savebox{\boxsigmambogen}
   { \begin{picture}(1,3)
       \put(0,0){\line(0,1){0.5}}
       \put(0,1.5){\oval(0.5,2)[r]}
       \put(0,2.5){\line(0,1){0.5}}
     \end{picture}} 
\newsavebox{\boxsigmam}
   \savebox{\boxsigmam}
   { \begin{picture}(9,3)
      \put(1.1,3){$\circ$}
      \put(2.1,3){$\circ$}
      \put(3.1,3){$\circ$}
      \put(7.1,3){$\circ$}                  
      \put(8.1,3){$\circ$}
      \put(9.1,3){$\circ$}                  
      \put(0,0){\uppartii{1}{1}{9}}
      \put(4.8,2.3){$\cdots$}      
      \put(0,1){\upparti{1}{5}}
      \put(2,0){\usebox{\boxsigmambogen}}
      \put(3,0){\usebox{\boxsigmambogen}}
      \put(7,0){\usebox{\boxsigmambogen}}
      \put(8,0){\usebox{\boxsigmambogen}}                  
      \put(0,4){\partii{1}{1}{9}}  
      \put(4.8,0){$\cdots$}           
      \put(1.1,-0.3){$\circ$}
      \put(2.1,-0.3){$\circ$}
      \put(3.1,-0.3){$\circ$}
      \put(7.1,-0.3){$\circ$}                  
      \put(8.1,-0.3){$\circ$}
      \put(9.1,-0.3){$\circ$}                  
     \end{picture}}        
\begin{center}
\begin{picture}(22,4)
 \put(0,1.5){$\pi_m=$}
 \put(1,0.2){\usebox{\boxpim}}
 \put(11,1.5){$\sigma_m=$}
 \put(12,0.2){\usebox{\boxsigmam}}
\end{picture}
\end{center}
\end{lem}
\begin{proof}
Let $n\in\N$ and let $X\subset P$ be $n$-admissible. Let $p\in H_n(X)\cap P(0,l)\neq\emptyset$  and $q\in H_n(X)\cap P(k,0)\neq\emptyset$ for some $k,l\in\N$. 

(1) We first prove that $\vierpartrot\in H_n(X)$ implies $\paarpart,\baarpart\in H_n(X)$. Composing $p$ iteratively with suitable tensor products of $\vierpartrot$ and $\idpart$, we infer that the partition $b_l\in P(0,l)$ consisting in a single block on $l$ points is in $H_n(X)$, see also \cite[Lemma B.2]{PartI}. Thus $\sum_k u_{ik}^l=1$ in $A_n(X)$, using relations $R(b_l)$, see also \cite[App. A]{PartI}. Moreover, recall from \cite[App. A]{PartI} that we have:
\[R(\vierpartrot): \qquad u_{ik}u_{jk}=u_{ki}u_{kj}=0\quad \textnormal{if }i\neq j\]
So, all we have to prove is $\sum_k u_{ik}^2=\sum_k u_{kj}^2=1$ in order to verify that the relations $R(\paarpart)$ and $R(\baarpart)$ hold. Put $z_i:=\sum_k u_{ik}$. Then $z_i$ is selfadjoint and we have, using $R(\vierpartrot)$ and $R(b_l)$:
\[z_i^l=\sum_{k_1,\ldots,k_l}u_{ik_1}\ldots u_{ik_l}=\sum_k u_{ik}^l =1\]
By the Idempotents Lemma \ref{LemIdem}, we infer $1=z_i^2=\sum_{k_1,k_2}u_{ik_1}u_{ik_2}=\sum_k u_{ik}^2$. Likewise we prove $\sum_k u_{kj}^2=1$.

(2) For a general partition $\pi_m\in H_n(X)$, note that composing $p^{\otimes m}$ with suitable tensor products of $\pi_m$ and $\idpart$ yields $b_{lm}\in H_n(X)$. By shifted doubling on two legs (Lemma \ref{LemWeakDoubling}) we obtain $\vierpartrot\in H_n(X)$ and we are back to Step (1).

(3) We observe that $\sigma_m\in H_n(X)$ means that the following relations hold in $A_n(X)$:
\[R(\sigma_m):\quad  u_{ik}u_{a_1b_1}u_{a_2b_2}\ldots u_{a_{m-2}b_{m-2}}u_{jk}=u_{ki}u_{a_1b_1}u_{a_2b_2}\ldots u_{a_{m-2}b_{m-2}}u_{kj}=0\quad \textnormal{if }i\neq j\]
Now, given $i\neq j$, we put $u_{a_{m-2}b_{m-2}}=u_{jk}$ and thus:
\[ u_{ik}u_{a_1b_1}u_{a_2b_2}\ldots u_{a_{m-3}b_{m-3}}u_{jk}u_{jk}=0\]
Therefore,
\[ \left(u_{ik}u_{a_1b_1}u_{a_2b_2}\ldots u_{a_{m-3}b_{m-3}}u_{jk}\right)\left(u_{jk} u_{a_{m-3}b_{m-3}}\ldots u_{a_2b_2}u_{a_1b_1}u_{ik} \right)=0,\]
which implies that the relations $R(\sigma_{m-1})$ hold, after a similar computation with $u_{a_{m-2}b_{m-2}}=u_{kj}$. Inductively, we obtain the relations $R(\sigma_2)=R(\vierpartrot)$ and we are back to Step (1).
\end{proof}

The next lemma is a technical generalization of the previous lemma, which is needed for Lemma \ref{LemThreeBlock}. 

\begin{lem}\label{LemProjIsPO}
Let $q\in P(s,s)$ be a partition such that 
\begin{itemize}
\item[(i)] $q=q^*=q^2$,
\item[(ii)] the upper left, the upper right, the lower left and the lower right points of $q$ form a block of size four,
\item[(iii)] all other blocks of $q$ are of size two.
\end{itemize}
Then $q\in P_O$.
\end{lem}
\begin{proof}
Let $n\in\N$, let $X\subset P$ be $n$-admissible with $H_n(X)\cap P(0,l)\neq\emptyset$ and $H_n(X)\cap P(k,0)\neq\emptyset$ for some $k,l$. Assume $q\in H_n(X)$. We shall prove that $\sigma_m\in H_n(X)$ holds for some $m$, which together with Lemma \ref{LemPisigma} implies $q\in P_O$. By Lemma \ref{LemProjPart} and assumption (iii), we know that if a block of $q$ connects an $i$-th upper point, $1<i<s$, to a $j$-th lower point, then $i=j$. Such a block will be called a \emph{pair through block}. If all blocks of size two are pair through blocks, then $q=\sigma_s$ and we are done. So, assume $q\neq \sigma_s$ and let $1\leq r'<s$ be the smallest number such that the block containing $r'+1$ is not a pair through block.
We may assume the following minimality for $q$ amongst all partitions in $H_n(X)$.
\begin{itemize}
\item[(M1)] The number $s$ is minimal amongst all partitions in $H_n(X)$ satisfying (i), (ii) and (iii).
\item[(M2)] The number $r$ of pair through blocks is minimal amongst all partitions $p$ in $H_n(X)\cap P(s,s)$ satisyfying (M1).
\item[(M3)] The number $r'$ is minimal amongst all partitions satisfying (M2).
\end{itemize}
We define:
\[q':=\left(\idpart^{\otimes r'}\otimes q\right)\left(q\otimes\idpart^{\otimes r'}\right)\left(\idpart^{\otimes r'}\otimes q\right)\]
Then $q'\in H_n(X)$ by Theorem \ref{ThmPartI}. 

We require some technical notation in order to study the above composition 
 in detail.
\begin{itemize}
\item Let $h^{\textnormal{up}}\in P(s-1,s-1)$ be the partition obtained from restricting $\idpart^{\otimes r'}\otimes q$ to the points $\{r'+1,\ldots,r'+s-1\}$. Note that all blocks of $h^{\textnormal{up}}$ are of size two.
\item Let $h_{\textnormal{down}}\in P(s-1,s-1)$ be the partition obtained from restricting $q\otimes\idpart^{\otimes r'}$ to the points $\{r'+1,\ldots,r'+s-1\}$. Again, all blocks of $h_{\textnormal{down}}$ are of size two.
\item Let $1<a_1,\ldots,a_v\leq s-1$ be the upper points of $h^{\textnormal{up}}$ belonging to pair through blocks. Note that $a_o:=1$ also belongs to a pair through block by assumption (ii).
\item Let $b_0:=s-r'$. This upper point of $h_{\textnormal{down}}$ belongs to a pair through block by assumption (ii).
Let $1<b_1,\ldots,b_w\leq s-1$ be the remaining upper points of $h_{\textnormal{down}}$ belonging to pair through blocks. 
\end{itemize}
Since $r'-1$ pair through blocks of $q$ sit on the points $2,\ldots,r'$, by definition $h_{\textnormal{down}}$ has the same number of pair through blocks as $h^{\textnormal{up}}$, so $v=w$. The above notation as well as the strategy of the proof is illustrated by the following picture:
   \savebox{\boxsigmambogen}
   
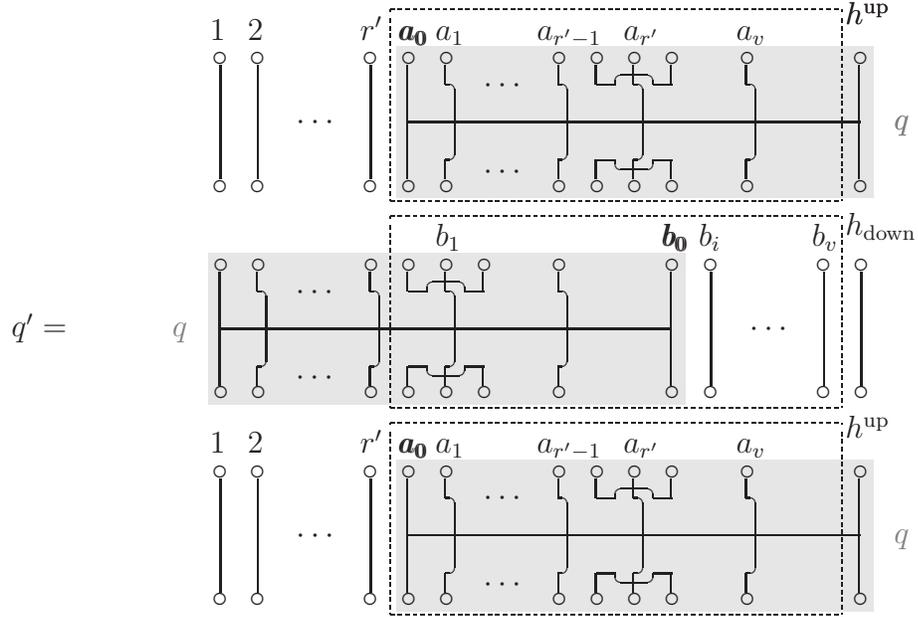
{ \begin{picture}(1,3)
       \put(0,0){\line(0,1){0.5}}
       \put(0,1.5){\oval(0.5,2)[r]}
       \put(0,2.5){\line(0,1){0.5}}
     \end{picture}} 
\newsavebox{\boxuppaarbogen}
   \savebox{\boxuppaarbogen}
   { \begin{picture}(2,1)
       \put(0,0){\line(1,0){0.5}}
       \put(1,0){\oval(1,0.5)[t]}
       \put(1.5,0){\line(1,0){0.5}}
     \end{picture}} 
\newsavebox{\boxuppaar}
   \savebox{\boxuppaar}
   { \begin{picture}(2,1)
      \put(0,0){\line(0,1){0.5}}
      \put(-0.25,0){\usebox{\boxuppaarbogen}}
      \put(2,0){\line(0,1){0.5}}
     \end{picture}} 
\newsavebox{\boxdownpaarbogen}
   \savebox{\boxdownpaarbogen}
   { \begin{picture}(2,1)
       \put(0,0){\line(1,0){0.5}}
       \put(1,0){\oval(1,0.5)[b]}
       \put(1.5,0){\line(1,0){0.5}}
     \end{picture}} 
\newsavebox{\boxdownpaar}
   \savebox{\boxdownpaar}
   { \begin{picture}(2,1)
      \put(0,0){\line(0,1){0.5}}
      \put(-0.25,0.5){\usebox{\boxdownpaarbogen}}
      \put(2,0){\line(0,1){0.5}}
     \end{picture}} 
\newsavebox{\boxqsigmam}
   \savebox{\boxqsigmam}
   { \begin{picture}(14,4)
      \put(1,3.3){\colorbox{hellgrau}{\qquad\qquad\qquad\qquad\qquad\qquad\quad\quad\;\;\;}}
      \put(1,2.9){\colorbox{hellgrau}{\qquad\qquad\qquad\qquad\qquad\qquad\quad\quad\;\;\;}}      
      \put(1,2.5){\colorbox{hellgrau}{\qquad\qquad\qquad\qquad\qquad\qquad\quad\quad\;\;\;}}
      \put(1,2.1){\colorbox{hellgrau}{\qquad\qquad\qquad\qquad\qquad\qquad\quad\quad\;\;\;}}      
      \put(1,1.7){\colorbox{hellgrau}{\qquad\qquad\qquad\qquad\qquad\qquad\quad\quad\;\;\;}}
      \put(1,1.3){\colorbox{hellgrau}{\qquad\qquad\qquad\qquad\qquad\qquad\quad\quad\;\;\;}}      
      \put(1,0.9){\colorbox{hellgrau}{\qquad\qquad\qquad\qquad\qquad\qquad\quad\quad\;\;\;}}
      \put(1,0.5){\colorbox{hellgrau}{\qquad\qquad\qquad\qquad\qquad\qquad\quad\quad\;\;\;}}      
      \put(1,0.1){\colorbox{hellgrau}{\qquad\qquad\qquad\qquad\qquad\qquad\quad\quad\;\;\;}} 
      \put(1,-0.3){\colorbox{hellgrau}{\qquad\qquad\qquad\qquad\qquad\qquad\quad\quad\;\;\;}}                

      \put(1.1,3){$\circ$}
      \put(2.1,3){$\circ$}
      \put(5.1,3){$\circ$}
      \put(6.1,3){$\circ$}                  
      \put(7.1,3){$\circ$}
      \put(8.1,3){$\circ$}                  
      \put(10.1,3){$\circ$}                        
      \put(13.1,3){$\circ$}   
      \put(0,0){\upparti{3}{1}}
      \put(0,0){\upparti{3}{13}}
      \put(1.3,1.5){\line(1,0){12}}
      \put(3.3,2.3){$\cdots$}      
      \put(2,0){\usebox{\boxsigmambogen}}
      \put(5,0){\usebox{\boxsigmambogen}}
      \put(7,0){\usebox{\boxsigmambogen}}
      \put(10,0){\usebox{\boxsigmambogen}}   
      \put(6,2.5){\usebox{\boxuppaar}}
      \put(6,0){\usebox{\boxdownpaar}}      
      \put(3.3,0){$\cdots$}    
      \put(1.1,-0.4){$\circ$}
      \put(2.1,-0.4){$\circ$}
      \put(5.1,-0.4){$\circ$}
      \put(6.1,-0.4){$\circ$}                  
      \put(7.1,-0.4){$\circ$}
      \put(8.1,-0.4){$\circ$}                  
      \put(10.1,-0.4){$\circ$}
      \put(13.1,-0.4){$\circ$}    
     \end{picture}}        
\newsavebox{\boxhup}
   \savebox{\boxhup}
   { \begin{picture}(19,4)
 \put(1.05,9){$\circ$}
 \put(2.05,9){$\circ$}
 \put(5.05,9){$\circ$}  
 \put(0,6){\upparti{3}{1}}
 \put(0,6){\upparti{3}{2}}
 \put(3.3,7.3){$\cdots$}
 \put(0,6){\upparti{3}{5}}  
 \put(1.05,5.6){$\circ$}
 \put(2.05,5.6){$\circ$}
 \put(5.05,5.6){$\circ$}  
 \put(4.7,6){\usebox{\boxqsigmam}}
 \put(1,9.7){1}
 \put(2,9.7){2}
 \put(5,9.7){$r'$} 
 \put(6,9.7){$a_0$}
  \put(6.05,9.7){$a_0$}
 \put(7,9.7){$a_1$}
 \put(9.7,9.7){$a_{r'-1}$}
 \put(12,9.7){$a_{r'}$}
 \put(15,9.7){$a_v$}    
 \put(19.2,7.3){{\color{hgrau} $q$}}   
 \put(5.8,5.4){\dashbox{0.1}(12,5.1){}} 
 \put(17.9,10.1){$h^{\textnormal{up}}$}
     \end{picture}}        
\newsavebox{\boxhdown}
   \savebox{\boxhdown}
   { \begin{picture}(19,4)
 \put(-0.3,0.5){\usebox{\boxqsigmam}}
 \put(14.1,3.5){$\circ$}
 \put(17.1,3.5){$\circ$}
 \put(18.1,3.5){$\circ$}
 \put(1.05,0.5){\upparti{3}{13}}
 \put(15.35,1.8){$\cdots$}
 \put(1.05,0.5){\upparti{3}{16}}
 \put(1.05,0.5){\upparti{3}{17}}  
 \put(14.1,0.1){$\circ$}
 \put(17.1,0.1){$\circ$}
 \put(18.1,0.1){$\circ$} 
 \put(19.2,7.3){{\color{hgrau} $q$}}   
 \put(5.8,5.4){\dashbox{0.1}(12,5.1){}} 
 \put(17.9,10.1){$h^{\textnormal{up}}$}
 \put(7,4.2){$b_1$}
 \put(13,4.2){$b_0$}
  \put(13.05,4.2){$b_0$}  
 \put(14,4.2){$b_i$}
 \put(17,4.2){$b_v$} 
 \put(5.8,-0.1){\dashbox{0.1}(12,5.1){}}  
 \put(17.9,4.5){$h_{\textnormal{down}}$}
 \put(0,1.8){{\color{hgrau} $q$}} 
     \end{picture}}        
\begin{center}
\begin{picture}(23,17)
 \put(0,7.8){$q'=$}
 \put(4,6){\usebox{\boxhup}}
 \put(4,6){\usebox{\boxhdown}} 
 \put(4,-5){\usebox{\boxhup}}
\end{picture}
\end{center}

Since $h^{\textnormal{up}}$ and $h_{\textnormal{down}}$ are pair partitions (i.e. all blocks are of size two), their composition is a pair partition, too. Hence, $a_0$ is either connected to some $a_i$, or to some $b_i$ throughout the composition of $h^{\textnormal{up}}$ and $h_{\textnormal{down}}$.

\emph{Step 1.} If $a_0$ was connected to some $a_i$, the $(r'+1)$-th upper point and the upper point $a_i+r'$ of $q'$ would be in the same block, as well as the $(r'+1)$-th lower point and the lower point $a_i+r'$ of $q'$. Using a weak restriction as in Section \ref{SectDoubling}, we would obtain a partition $q''\in P(a_i,a_i)$ in $H_n(X)$ satisfying (i), (ii) and (iii), in contradiction to the minimality condition (M1).

\emph{Step 2a.} If $a_0$ was connected to some $b_i$, $i\neq 0$, and some $a_i$ was connected to some $a_j$, this would yield a partition in $H_n(X)$ contradicting (M2).

\emph{Step 2b.} If $a_0$ was connected to some $b_i$, $i\neq 0$, and no $a_i$ was connected to some $a_j$, then some $a_i$ was connected to $b_0$. Then, $q'$ would be a partition such that the left upper point and the upper point $a_i+r'$ are in the same block with their counterparts on the lower line; hence we would have a  contradiction to (M1) (in case $a_i+r'$ is on a position smaller than $s$) or to (M3) (in case $a_i+r'$ is on the $s$-th position and then cutting all points of the right of $s$, turning the block of the $(r'+1)$-point into a pair through block).

\emph{Step 3.} We infer that $a_0$ must be connected to $b_0$, which yields $\sigma_{r'+1}\in H_n(X)$ when weakly restricting (Section \ref{SectDoubling}) $q'$ to its first $r'+1$ points.
\end{proof}

\begin{lem}
\label{LemThreeBlock}
Let $p\in P(0,l)$ be a partition containing no singletons, but containing a block of length at least three. Then $p\in P_O$.
\end{lem}
\begin{proof}
Let $n\in\N$, let $X\subset P$ be $n$-admissible and assume $p,p^*\in H_n(X)$.
By weak rotation (Lemma \ref{LemWeakRot}), we may assume that the points $1,s$ and $t$ with $1<s<t\leq l$ are in the same block (possibly containing further points). 
Moreover, we may assume that $s$ is minimal in the sense that $1$ (and hence also $s$) is connected to no point $1<x<s$, and also there are no three points $a,b,c$ in the same block with $1<a<b< s$. 

By partial doubling of $p$ (Lemma \ref{LemPartDoubling}), there is a partition $q\in P(s,s)$ in $H_n(X)$ such that the upper points $1$ and $s$ and the lower points $1$ and $s$ are all in the same block $V$. We prove that $V$ consists of exactly these four points and that all other blocks of $q$ consist in exactly two points. Indeed, as $p$ has no singletons, $q$ must have no singletons. Now, pick a lower point $1<t<s$. If it is linked in $p$ to another lower point, it cannot be linked to any upper point nor to a third lower point, by the minimality assumption on $s$. We infer that all blocks of $q$ have size exactly two, apart from the block $V$ which has exactly four points. Together with Lemmata \ref{LemPartDoublingProj} and \ref{LemProjIsPO}, we obtain $q\in P_O$. Since $q\in H_n(X)$ and $q=q^*$, we infer $\paarpart,\baarpart\in H_n(X)$.
\end{proof}

\begin{rem}
\label{RemNCm}
The above lemma shows that although the sets $NC_{[m]}$ of Proposition \ref{PropExNonBSCateg} are generalized categories of partitions which are \emph{no} Banica-Speicher categories of partitions, the sets $H_n(NC_{[m]})$ in turn \emph{are} Banica-Speicher categories of partitions (Proposition \ref{PropEnforcedOrth}). So, the associated partition $C^*$-algebras $A_n(NC_{[m]})$ are in fact Banica-Speicher quantum groups.
\end{rem}

\subsection{Blocks of length one}

\begin{lem}\label{LemPOSingleton}
Let $X\subset P$ be $n$-admissible and assume $H_n(X)\cap P(0,l)\neq\emptyset$ and $H_n(X)\cap P(k,0)\neq\emptyset$ for some $k,l\in\N$. If $H_n(X)$ contains any partition $h\in P(k',l')$ which has a singleton block, then also $\singleton\otimes\singleton,\downsingleton\otimes\downsingleton\in H_n(X)$. If $l$ or $k$ is odd, then even $\singleton,\downsingleton\in H_n(X)$ holds true.
\end{lem}
\begin{proof}
If $h$ has a singleton block on its lower line, we choose a partition $p\in H_n(X)\cap P(0,l)$ and consider the following partition $p'$ obtained by composition:
\[p':=h^{\otimes l}p^{\otimes k'}\in P(0,l'l)\]
Thus, $p'$ contains a block which is a singleton and by weak rotation (Lemma \ref{LemWeakRot}), we may assume that the first point of $p'$ is a singleton block. By shifted doubling (Lemma \ref{LemWeakDoubling}), we infer $\idpartsingleton\in H_n(X)$. Likewise, we deduce $\idpartsingleton\in H_n(X)$ if $h$ has a singleton block on its upper line, using a partition $q\in H_n(X)\cap P(k,0)$. The relations $R(\idpartsingleton)$ allow us to define $z:=\sum_k u_{ik}=\sum_ku_{kj}$ independent of the choice of $i$ or $j$, see \cite[App. A]{PartI}. Now, by Lemma \ref{LemBasicOper}(b), we conclude that $\singleton^{\otimes l},\downsingleton^{\otimes k}\in H_n(X)$ and hence $z^l=1$ and $z^k=1$ using the relations $R(\singleton^{\otimes l})$ or rather $R(\downsingleton^{\otimes k})$. As $z=z^*$, we use the Idempotents Lemma \ref{LemIdem} to finish the proof.
\end{proof}

\begin{lem}\label{LemSingletonCase}
Let $p\in P(0,l)$ be a partition containing a singleton as a block and let $l$ be even.
\begin{itemize}
\item[(a)] If $p$ contains two points $1\leq s<t\leq l$ which are in the same block and  $t-s$ is odd, then $p\in P_O$.
\item[(b)] If $p$ contains two points $1\leq s<t\leq l$ which are in the same block and  $t-s$ is even, then $p\in H_n(X)$ implies $\positioner\in H_n(X)$ for all $n$-admissible sets  $X\subset P$ with  $H_n(X)\cap P(k,0)\neq\emptyset$ for some $k\in\N$. 
\end{itemize}
\end{lem}
\begin{proof}
Let $X\subset P$ be $n$-admissible and assume $H_n(X)\cap P(k,0)\neq\emptyset$ for some $k\in\N$. Assume $p\in H_n(X)$. By Lemma \ref{LemPOSingleton}, we have $\downsingleton\otimes\downsingleton\in H_n(X)$. Composing $p$ with suitable tensor powers of $\downsingleton\otimes\downsingleton$ and $\idpart$, we infer $\paarpart\in H_n(X)$ in Case (a) and $\positioner\in H_n(X)$ in Case (b) (possibly using Theorem \ref{ThmPartI}(d) in the latter case). From $p^*\in H_n(X)$, we deduce $\baarpart\in H_n(X)$ in Case (a) in a similar way.
\end{proof}

\subsection{Blocks of size two}

\begin{lem}\label{LemPaarBaarPO}
The partition $\paarbaarpart$ is in $P_O$.
\end{lem}
\begin{proof}
Let $n\in\N$, let $X\subset P$ be $n$-admissible and assume $H_n(X)\cap P(0,l)\neq\emptyset$ and $H_n(X)\cap P(k,0)\neq\emptyset$ for some $k,l\in\N$. Note that we may assume that $k$ and $l$ are multiples of four, by Theorem \ref{ThmPartI}(b).  Assume $\paarbaarpart\in H_n(X)$. Composing any partition $p\in P(0,4l)$ with $\paarbaarpart^{\otimes 2l}$, we infer $\paarpart^{\otimes 2l}\in H_n(X)$. We define:
\[X_{ij}:=\sum_k u_{ik}u_{jk}\in A_n(X)\]
We compute, using the relation $R(\paarpart^{\otimes 2l})$:
\begin{align*}
(X_{ij}X_{ij}^*)^l&=\sum_{k_1,\ldots,k_{2l}}u_{ik_1}u_{jk_1}u_{jk_2}u_{ik_2}\ldots u_{ik_{2l-1}}u_{jk_{2l-1}}u_{jk_{2l}}u_{ik_{2l}}\\
&=\sum_{\gamma_1',\ldots,\gamma_{4l}'}\delta_{\paarpart^{\otimes 2l}}(\emptyset,\gamma')u_{i\gamma_1'}u_{j\gamma_2'}u_{j\gamma_3'}u_{i\gamma_4'}\ldots u_{i\gamma_{4l-3}'}u_{j\gamma_{4l-2}'}u_{j\gamma_{4l-1}'}u_{i\gamma_{4l}'}\\
&=\delta_{\paarpart^{\otimes 2l}}(\emptyset,(ijji\ldots ijji))\\
&=\delta_{ij}
\end{align*}
For $i\neq j$, this proves $X_{ij}X_{ij}^*=0$ and hence $X_{ij}=0$. On the other hand, we have $X_{ii}^{2l}=1$, which implies $X_{ii}=1$ by the Idempotents Lemma \ref{LemIdem}. We conclude that $\paarpart\in H_n(X)$ holds, and we prove $\baarpart\in H_n(X)$ in a similar way.
\end{proof}

\begin{lem}\label{LemPaarCase}
Let $X\subset P$ be $n$-admissible and assume $H_n(X)\cap P(k,0)\neq\emptyset$ for some $k\in\N$.
Let $p\in P(0,l)$ be a partition consisting only in blocks of size two. If $p\in H_n(X)$, then there is a number $m\geq 2$ such that the  following partition $\tau_m\in P(m,m)$ is in $H_n(X)$.
   \savebox{\boxsigmambogen}
   
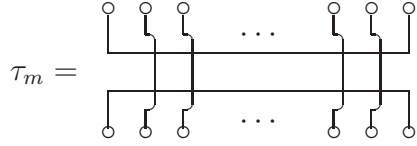
{ \begin{picture}(1,3)
       \put(0,0){\line(0,1){0.5}}
       \put(0,1.5){\oval(0.5,2)[r]}
       \put(0,2.5){\line(0,1){0.5}}
     \end{picture}} 
\newsavebox{\boxtaum}
   \savebox{\boxtaum}
   { \begin{picture}(9,3)
      \put(1.1,3){$\circ$}
      \put(2.1,3){$\circ$}
      \put(3.1,3){$\circ$}
      \put(7.1,3){$\circ$}                  
      \put(8.1,3){$\circ$}
      \put(9.1,3){$\circ$}                  
      \put(0,0){\uppartii{1}{1}{9}}
      \put(4.8,2.3){$\cdots$}      
      \put(2,0){\usebox{\boxsigmambogen}}
      \put(3,0){\usebox{\boxsigmambogen}}
      \put(7,0){\usebox{\boxsigmambogen}}
      \put(8,0){\usebox{\boxsigmambogen}}                  
      \put(0,4){\partii{1}{1}{9}}  
      \put(4.8,0){$\cdots$}           
      \put(1.1,-0.3){$\circ$}
      \put(2.1,-0.3){$\circ$}
      \put(3.1,-0.3){$\circ$}
      \put(7.1,-0.3){$\circ$}                  
      \put(8.1,-0.3){$\circ$}
      \put(9.1,-0.3){$\circ$}                  
     \end{picture}}        
\begin{center}
\begin{picture}(12,4)
 \put(0,1.5){$\tau_m=$}
 \put(1,0.2){\usebox{\boxtaum}}
\end{picture}
\end{center}
\end{lem}
\begin{proof}
The first point of $p$ is connected to some point $m$. We may assume that $m$ is minimal, by weak rotation (Lemma \ref{LemWeakRot}), i.e. there are no two points $1<s<t<m$ such that $s$ and $t$ are in the same block. By partial doubling (Lemma \ref{LemPartDoubling}) we obtain $\tau_m\in H_n(X)$.
\end{proof}

\subsection{Partitions of odd length}

\begin{lem}\label{LemOddLength}
Let $p\in P(0,l)$ be such that $l$ is odd. If $p\neq \singleton^{\otimes l}$, then $p\in P_O$.
\end{lem}
\begin{proof}
Let $n\in\N$, let $X\subset P$ be $n$-admissible and assume $p,p^*\in H_n(X)$.
If $p$ contains a singleton, Lemma \ref{LemPOSingleton} yields that $\singleton,\downsingleton\in H_n(X)$. As $p\neq\singleton^{\otimes l}$, there are two points $1\leq s<t\leq l$ which are in the same block. Composing $p$ with 
\[\downsingleton^{\otimes (s-1)}\otimes\idpart\otimes\downsingleton^{\otimes (t-s-1)}\otimes\idpart\otimes\downsingleton^{\otimes (l-t)},\]
we obtain $\paarpart\in H_n(X)$. Likewise, the composition of $p^*$ with tensor powers of $\singleton$ and $\idpart$ yields $\baarpart\in H_n(X)$. If $p$ does not contain a singleton, we find a block of length at least three, since $l$ is odd, and we use Lemma \ref{LemThreeBlock} in order to finish the proof.
\end{proof}

\subsection{Cases excluded from our machinery}

If $H_n(X)$ is such that $H_n(X)\cap P(0,l)=\emptyset$ and $H_n(X)\cap P(k,0)=\emptyset$ for all $k,l\in\N$, we do not know whether $(A_n(X),u)$ is a compact matrix quantum group (the invertibility of $u$ is unclear), see Theorem \ref{ThmExCMQG1}. Hence, we refrained from investigating $P_O$ in this case. The same holds true, for our general assumption $p\neq\singleton^{\otimes l}$ -- in the opposite case, again the compact matrix quantum group structure is not guaranteed.

\begin{ex}
 Let us end this section with an example on relations which do not imply orthogonality of $u$. Let $a\in\mathbb C$ with $a\neq 0$ and $a\neq 1$. The matrix
\[\begin{pmatrix} a&1-a&0\\1-a&0&a\\0&a&1-a\end{pmatrix}\]
satisfies the relations $R(\singleton)$ and $R(\downsingleton)$, but it is not orthogonal. However, the set $\{\singleton,\downsingleton\}$ is not $n$-admissible, since $a$ may be chosen arbitrarily large. We thank Stefan Jung for communicating this example to us.
\end{ex}

\section{Summary and non-unitary Banica-Speicher quantum groups}
\label{SectSummary}

Summarizing our investigations of $P_O$, we obtain the following theorem characterizing all cases of partitions  $p\in P(0,l)$ with $p\neq \singleton^{\otimes l}$.

\begin{thm}\label{ThmCriteria}
Let $X\subset P$ be $n$-admissible and assume  $H_n(X)\cap P(k,0)\neq\emptyset$ for some $k\in\N$. Let $p\in P(0,l)$ be such that $p\neq \singleton^{\otimes l}$.
\begin{itemize}
\item[(a)] If $l$ is odd, then $p\in P_O$.
\item[(b)] If $l$ is even and if $p$ contains a singleton as well as two points $1\leq s<t\leq l$ which are in the same block and  $t-s$ is odd, then $p\in P_O$.
\item[(c)] If $l$ is even and if $p$ contains a singleton as well as two points $1\leq s<t\leq l$ which are in the same block and  $t-s$ is even, then $p\in H_n(X)$ implies $\positioner\in H_n(X)$.
\item[(d)] If $l$ is even and if $p$ contains no singletons but a block of length at least three, then $p\in P_O$.
\item[(e)] If $l$ is even and if all blocks of $p$ are of size two, then $p\in H_n(X)$ implies $\tau_m\in H_n(X)$ for some $m\geq 2$. Note that $\tau_2=\paarbaarpart\in P_O$.
\end{itemize}
\end{thm}
\begin{proof}
(a) Lemma \ref{LemOddLength}

(b) and (c) Lemma \ref{LemSingletonCase}

(d) Lemma \ref{LemThreeBlock}

(e) Lemmata \ref{LemPaarBaarPO} and \ref{LemPaarCase}
\end{proof}

We conclude that in many cases the orthogonality of the matrix $u$ is implied by partitions $p\in H_n(X)$, even for $p\notin\{\paarpart,\baarpart\}$. In these cases, the compact matrix quantum groups associated to the partition $C^*$-algebras are nothing but Banica-Speicher quantum groups -- although we might have started with a set $X\subset P$ which is not a Banica-Speicher category of partitions.

On the other hand, we isolated a few partitions $p\in P(0,l)$, which imply that $(A_n(X),u)$ is a compact matrix quantum group, but for which we were not able to prove $p\in P_O$. Conversely, if one can show $p\notin P_O$, then this would yield an example of a ``non-unitary Banica-Speicher quantum group'', i.e. a quantum group arising from partitions which does not fit into the definition of Banica and Speicher, because the matrix $u$ is not orthogonal (and even not unitary). These candidates for non-unitary Banica-Speicher quantum groups $(A_n(\{p,p^*\}),u)$ are:
\begin{itemize}
\item[(P1)] $p=\positioner$
\item[(P2)] $p=\crosspartoneline$ (implying $\tau_3\in H_n(X)$)
\item[(P3)] More exotic partitions falling into the classes (c) and (e) of Theorem \ref{ThmCriteria}.
\end{itemize}

What is needed now is either a proof of $p\in P_O$ for the partitions in (P1), (P2) and (P3), or conversely a proof that the matrix $u$ in $(A_n(\{p,p^*\}),u)$ is not orthogonal. For the latter, one probably needs to find a good representation of $A_n(\{p,p^*\})$ on some Hilbert space, but unfortunately we have none at hand.  For proving $p\notin P_O$ in case (P2) it would be enough to prove that $A_n(\{p,p^*\})$ is noncommutative. See also \cite[Problem 3.5 and 3.6]{PartI} for a concrete formulation of the problem regarding (P1) and (P2), in terms of basic Hilbert space language.

\bibliographystyle{alpha}
\bibliography{PartitionBib}

\end{document}